\documentclass[english]{article}
\usepackage[T1]{fontenc}
\usepackage[latin9]{inputenc}
\usepackage{babel}
\usepackage{amsmath,amsfonts,amssymb,latexsym,epsfig, amsthm}
\usepackage{graphicx,color}
\usepackage{hyperref}
\usepackage{setspace}
\usepackage[authoryear]{natbib}

\usepackage{authblk}
\usepackage{xcolor}
\hypersetup{
    colorlinks=true,
    citecolor =blue,
    linkcolor=red,
    urlcolor=cyan,
}

\newtheorem{assumption}{Assumption}

\newtheorem{prop}{Proposition}
\newtheorem{cor}{Corollary}
\newtheorem{define}{Definition}
\newtheorem{rem}{Remark}

\newtheorem{theorem}{Theorem}

\newcommand{\RR}{\mathbb{R}}

\newcommand{\EE}{\mathbb{E}}

\newcommand{\lyxaddress}[1]{
\par {\raggedright #1
\vspace{1.4em}
\noindent\par}
}

\usepackage[textwidth=1.9cm, textsize=scriptsize]{todonotes}

\providecommand{\keywords}[1]{\textbf{\textit{Keywords}} #1}
\begin{document}

\title{Asymptotic Analysis of Model Selection Criteria for General
Hidden Markov Models}

\author{Shouto Yonekura$^{1,3}$, Alexandros Beskos$^{1,3}$, Sumeetpal S.Singh$^{2,3}$}
\maketitle

\lyxaddress{1 Department of Statistical Science, University College London, UK.\\
2 Department of Engineering, University of Cambridge, UK.\\
3 The Alan Turing Institute for Data Science, UK.}

\begin{abstract}
The paper obtains analytical results for the asymptotic properties 
of Model Selection Criteria -- widely used in practice -- for a general family
of hidden Markov models (HMMs), thereby substantially extending the related theory 
 beyond typical `i.i.d.-like' model structures and filling in an important gap 
 in the relevant literature.
In particular, we look at the Bayesian and Akaike Information Criteria (BIC and AIC) 
and the model evidence. In the setting of nested classes of models, we prove that BIC and the evidence are strongly consistent for HMMs (under regularity conditions), 
whereas AIC is not weakly consistent.
Numerical experiments support our theoretical
results. 
\end{abstract}



\section{Introduction}

Owning to their rich structure,   hidden Markov models (HMMs) are
routinely used in such diverse disciplines as finance
\citep{mamon2007hidden}, speech recognition~\citep{gales2008application},
epidemiology \citep{green2002hidden}, biology \citep{yoon2009hidden},
signal processing \citep{crouse1998wavelet}. We refer to \cite{del:04,del:13}
for a comprehensive study of the theory of HMMs and of accompanying Monte Carlo methods for their calibration to observations. Model Selection has been
one of the most well studied topics in Statistics. BIC \citep{schw:78} or AIC \citep{akai:74}  -- as well as their generalisations \citep{konishi1996generalised} -- and the evidence, are used in a wide range of contexts, including time series analysis \citep{shibata1976selection},
regression \citep{hurvich1989regression}, bias correction \citep{hurvich1990impact},
composite likelihoods \citep{varin2005note}.
For a comprehensive treatment of the subject of Model Selection, 
see e.g.~\cite{clae:08}. 

There has been relatively limited research on Model Selection for general classes of  HMMs used in practice. A fundamental property of a Model Selection Criterion
 is that of \emph{consistency} (to be defined analytically later on in the paper).
\cite{csiszar2000consistency}
prove strong consistency of BIC for discrete-time, finite-alphabet
Markov chains. In the HMM context, \cite{gass:03} consider discrete-time, finite-alphabet
HMMs and provide asymptotic and finite-sample analysis  of code-based and penalised maximum likelihood estimators (MLEs) using
tools from Information Theory and Stein's Lemma.
With regards to the  Bayesian approach to Model Selection, this
typically involves the marginal likelihood of the data (or evidence) \citep{jeffreys1998theory, kass1995bayes}. 
\cite{shao2017bayesian} show numerically that the evidence can be consistent for HMMs, 
however this has yet to be proven analytically.

The work in this paper makes a number of contributions, relevant for HMMs on general state spaces -- thus 
of wide practical significance and such that cover an important gap in the theory of HMMs
established in the existing literature. 
We remark that our analysis assumes smoothness conditions of involved functions w.r.t.~the parameter of interest, thus is intrinsically not relevant for interesting problems of discrete nature, an example being the identification of the number of states of the underlying Markov chain.
Our main results can be summarised as follows:
\begin{itemize}
\item[(i)] We establish sharp
asymptotic results (in the sense of obtaining $\limsup_n$ for the quantity of interest) for the log-likelihood function for HMMs evaluated at the MLE, in an a.s.~sense. 
A lot of the initial developments are borrowed from \cite{douc:14} (see also citations therein for 
more works on asymptotic properties of the MLE for HMMs). 
Moving from the study of the MLE to that of Model Selection Criteria is non-trivial, 
involving for instance use of the 
Law of Iterated Logarithm (LIL) for -- carefully developed -- martingales \citep{stou:70}.
\item[(ii)] We show that  BIC and the evidence are 
strongly consistent in the context of nested HMMs, whereas
AIC is not consistent. To the best
of our knowledge, this is the first time that such statements are proven in 
the literature for general HMMs. For AIC, we show that, w.p.~1, this criterion 
will occasionally choose the wrong model even under an infinite amount of information.
\end{itemize}

The rest of the paper is organised as follows. In Section  \ref{sec:asy}, we briefly
review some asymptotic results for the log-likelihood function and the  MLE without assuming model correctness. An important
departure from the i.i.d.~setting is that the log-likelihood function itself does not make up a stationary time-series
process even if the data are assumed to be derived from one. 
Section \ref{sec:asyL} begins with some asymptotic results for the MLE and the log-likelihood under model correctness. Later on, we move beyond the established literature and, by calling upon LIL for martingales,
we establish a number of fundamental asymptotic results, relevant for Model Selection Criteria. In Section \ref{sec:back}, we study the derivation of 
BIC (and its connection with the evidence) and AIC for general HMMs. 
In particular, an explicit result binding BIC and evidence will later on be used to show 
that the two criteria share similar consistency properties. 
Section \ref{sec:hmm} contains our main results.  We prove strong consistency
of BIC and the evidence and non-consistency of AIC for a class of nested HMMs.
Section ~\ref{sec:algorithm} reviews (for completeness) an algorithm borrowed from the literature, based on
Sequential Monte Carlo, for approximating AIC and BIC. 
We use this algorithm in Section \ref{sec:numerics} to present some numerical results that 
agree with our theory. We conclude in Section  \ref{sec:final}.

\section{Asymptotics under No-Model-Correctness}
\label{sec:asy}

We briefly summarise some asymptotic results for general HMMs needed in later
sections. The development follows closely \cite[Ch.~13]{douc:14}.
An HMM is a bivariate process $\{x_{k},z_{k}\}_{k\geq0}$
such that state component $\{x_{k}\}_{k\geq0}$ is an unobservable
Markov chain with initial law $x_0\sim \eta$ and transition kernel $Q_{\theta}(\cdot|x)$, with values 
in the measurable space $(\mathsf{X},\mathcal{X})$. 
We have adopted a parametric setting with $\theta\in\Theta\subseteq \mathbb{R}^{d}$,
for some $d\ge 1$.
Conditionally on $\{x_{k}\}_{k\geq 0},$
the distribution of the observation process instance $z_{k}$ depends only on $x_{k}=x$, independently over $k\ge 0$, and is given by the kernel
 $G_{\theta}(\cdot|x)$ defined on $(\mathsf{Y},\mathcal{Y})$. 
We assume that $\mathsf{X}$, $\mathsf{Y}$ are Polish spaces and $\mathcal{X}$, $\mathcal{Y}$ the corresponding Borel $\sigma$-algebras.
The notation $\{y_k\}_{k\ge 0}$ is reserved for the true \emph{data generating process}, 
 which may or may not belong in the parametric family of HMMs we specified above -- 
 meant to be distinguished from $\{z_{k}\}_{k\ge 0}$ which is the process driven by the model dynamics.
In particular, in this section we work under no-model correctness, i.e.~we do not have to assume 
the existence of a correct parameter value for the prescribed model that delivers the distribution 
of the data generating process. 
 
 Throughout the article, we assume that the
 following hold.

\begin{assumption}
	\label{ass:1}
The data generating process $\{y_{k}\}_{k\geq0}$ is
strongly stationary and ergodic. 
\end{assumption}

\begin{assumption}
	\label{ass:2}
\begin{itemize}
\item[(i)] $Q_{\theta}(\cdot|x)$
is absolutely continuous w.r.t.~a probability measure $\mu$
on $(\mathsf{X},\,\mathcal{X})$ with density $q_{\theta}(x'|x)$ --
$\mu$ is fixed 
for all $(x,\theta) \in\mathsf{X}\times \Theta$.

\item[(ii)] $G_{\theta}(\cdot|x)$
is absolutely continuous w.r.t.~a  measure $\nu$
on $(\mathsf{Y},\,\mathcal{Y})$ with density $g_{\theta}(y|x)$ --
$\nu$ is fixed  
for all $(x,\theta) \in\mathsf{X}\times \Theta$.

\item[(iii)] The initial distribution $\eta=\eta(dx_0)$ has density, denoted $\eta(x_0)$ -- with some abuse of notation --,
w.r.t.~$\mu$.


\item[(iv)] The parameter space $\Theta$ is a compact subset of $\mathbb{R}^{d}$; w.p.~1, $p_{\theta}(y_{0:n-1})>0$ for all $\theta\in\Theta$, 
 for all $n\ge 1$, where $p_{\theta}(\cdot)$ denotes here the density of the distribution of the observations under the model (for given $\theta$ and size $n$). 
\end{itemize}
\end{assumption}

\noindent 
Without loss of generality, we have assumed that $\eta(dx_0)$ does not depend on $\theta$.
Probability statements -- as in Assumption \ref{ass:2}(iv) -- and expectations throughout the paper are
to be understood w.r.t.~the law of the data generating process $\{y_k\}$.
Henceforth we make use of the notation $a_{i:j}=(a_i,\ldots,a_{j})$, for integers $i\le j$,
for a given sequence $\{a_k\}$.
We need the following conditions.

\begin{assumption}
	\label{ass:3}
There exist $\sigma^-$, $\sigma^+\in(0,\infty)$ so that  $$\sigma^{-}\leq q_{\theta}(x'|x)\leq\sigma^{+}$$
for any $x,\,x{'}\in\mathsf{X}$ and any $\theta\in\Theta$.
\end{assumption}

\noindent This is the strong mixing condition typically used in this context (e.g.~\cite{del:04, del:13}), 
providing a Dobrushin coefficient of $1-	\sigma_-/\sigma_+$ for the hidden Markov chain; it is critical for most of the results reviewed or developed in the sequel.
Assumption \ref{ass:3}  implies, for instance,  that for any $x\in\mathsf{X}$, $A\in\mathcal{X}$, $Q_{\theta}(A|x)\geq\sigma^{-}\mu(A)$,
that is, for any $\theta\in\Theta$, $\mathsf{X}$ is a
1-small set for the process $\{x_{k}\}_{k\ge 0}$. The chain has the
unique invariant measure $\pi_{\theta}^{X}$ and is uniformly ergodic,
so  for any $x\in\mathsf{X}$,  $n\geq0$,
$\left\Vert Q_{\theta}^{n}(\cdot|x)-\pi_{\theta}^{X}\right\Vert _{TV}\leq(1-\sigma^{-}/\sigma^+)^{n}$ --
with $\left\Vert\cdot \right\Vert _{TV}$ denoting total variation norm.


We calculate the likelihood and log-likelihood functions,
\begin{align}
p_{\theta}(y_{0:n-1}) & =\int
\eta(dx_0) g_{\theta}(y_{0}|x_0)\prod_{k=1}^{n-1}\big\{ q_{\theta}(x_{k}|x_{k-1})
 g_{\theta}(y_{k}| x_{k})\big\}\mu^{\otimes n}(dx_{1:n-1});
 \label{eq:calc}\\
\ell_{\theta}(y_{0:n-1}) & =\log p_{\theta}(y_{0:n-1})=\sum_{k=0}^{n-1}\log p_{\theta}(y_{k}| y_{0:k-1}).\label{eq:loglik}
\end{align}
Though $\{y_{k}\}_{k\geq0}$ is stationary and ergodic, terms $\{\log p_{\theta}(y_{k}| y_{0:k-1})\}_{k\ge 0}$
do not form a stationary process (in general). To obtain  stationary and ergodic log-likelihood terms, 
following  \cite{douc:04, capp:05, douc:14}, 
we work with the standard extension 
of the stationary $y$-process onto the whole of the integers,
and  write  
$\{y_{k}\}_{k=-\infty}^{\infty}$. 
One can then define the variable $\log p_{\theta}(y_{k}| y_{-\infty:k-1})$ as the a.s.~limit  of the Cauchy sequence (uniformly in $\theta$) 
$\log p_{\theta}(y_{k}| y_{-t:k-1})$ -- found as in (\ref{eq:calc}) for initial law $x_{-t}\sim \eta$ -- as $t\rightarrow \infty$; see \cite[Ch.~13]{douc:14} for more details.
%
We can now define the modified, stationary version of the  log-likelihood
\begin{align}
\ell_{\theta}^{s} & (y_{-\infty:n-1}):=\sum_{k=0}^{n-1}\log p_{\theta}(y_{k}| y_{-\infty:k-1}).\label{eq:extend}
\end{align}
%
%

\begin{assumption}
\label{ass:4}
We have that
%
$b^{+}:=\sup_{\theta\in\Theta}\sup_{x\in\mathcal{X},y\in\mathcal{Y}}g_{\theta}(y|x)<\infty$ 
and  $$\mathbb{E}\,|\log b^{-}(y_{0})|<\infty,$$
where $b^{-}(y):=\inf_{\theta\in\Theta}\int_{\mathsf{X}}g_{\theta}(y|x)\mu(dx)>0$.
\end{assumption}
\noindent The finite-moment part     implies 
that $\mathbb{E}\,|\log p_{\theta}(y_{0}| y_{-\infty:-1})|<\infty$, thus 
  Birkhoff's ergodic theorem can be applied for averages deduced from (\ref{eq:extend}).
\begin{prop}
\label{lem:station}
Under Assumptions \ref{ass:1}-\ref{ass:4},
\begin{align*}
\sup_{\theta\in\Theta}|\tfrac{1}{n}\ell_{\theta}(y_{0:n-1})-\tfrac{1}{n}\ell_{\theta}^{s}(y_{-\infty:n-1})| \le \tfrac{C}{n},
\end{align*}
for a constant $C>0$.
\end{prop}
\begin{proof}
This is Proposition 13.5 of \cite{douc:14}; the upper bound $C/n$ is implied from the proof of
that proposition.	
\end{proof}	
%
%
We consider the maximum likelihood estimator (MLE) defined as 
the set 
\begin{align}
\hat{\theta}_n = \arg\max_{\theta\in\Theta}\ell_{\theta}(y_{0:n-1}).\label{eq:MLE}
\end{align}
%

\begin{assumption}
\label{ass:5}

For all $(x,x{'})\in\mathsf{X}\times\mathsf{X}$ and $y\in\mathsf{Y}$,
the mappings $\theta\mapsto q_{\theta}(x{'}|x)$ and $\theta\mapsto g_{\theta}(y|x)$
are continuous.
\end{assumption}
\noindent Such requirements imply continuity of the log-likelihood   $\theta\mapsto \tfrac{1}{n}\ell_{\theta}(y_{0:n-1})$ and its limit $\theta \mapsto\mathbb{E}\,[\,\log p_{\theta}(y_{0}| y_{-\infty:-1})\,]$,
which -- together with other conditions -- provide convergence of the MLE to the maximiser of the limiting function.
For sets $A, B\subseteq \Theta$, we define $d(A,B):= 
\inf_{a\in A, b\in B}|a-b|$. 


\begin{prop}
	\label{prop:uniform}
	Under Assumptions \ref{ass:1}-\ref{ass:5} we have the following.
\begin{itemize}
\item[(i)] Let $\ell(\theta):=\mathbb{E}\,[\,\log p_{\theta}(y_0|y_{-\infty:-1})\,]$.
The function $\theta\mapsto \ell(\theta)$ is continuous, and we have  
\begin{equation*}
\lim_{n\rightarrow\infty}\sup_{\theta\in\Theta}
|\tfrac{1}{n}\ell_{\theta}(y_{0:n-1})-\ell(\theta)|	=0,\quad  w.p.\,\,1.
\end{equation*}
\item[(ii)] We have $\lim_{n\rightarrow\infty}d(\hat{\theta}_{n},\,\theta_{*})=0$, w.p.\,\,1,
where 
$$\theta_{*}:=\arg\max_{\theta\in\Theta}\ell(\theta)$$ 
is the set of global maxima of $\ell(\theta)$. 
\end{itemize}
\end{prop}
\begin{proof}
This is Theorem 13.7 of \cite{douc:14}. The proof of (i) is based on working with the
stationary version of the log-likelihood in (\ref{eq:extend}), permitted 
due to Proposition~\ref{prop:uniform}, and using Birkhoff's ergodic theorem.
\end{proof}
\noindent 
Recall that $\theta_*$ need not be thought of as the correct parameter value here, as no assumption of 
the class of HMMs containing the correct data-generating model 
is made in this section.	
To avoid identifiability 
issues, we make the following assumption on the HMM model.

\begin{assumption}
	\label{ass:six}	
	  $\theta_{*}$ is a singleton.
\end{assumption}

\noindent This implies immediately the following.

\begin{cor}
\label{cor:singleton}
The set of maxima $\hat{\theta}_n$ is a singleton for all large enough~$n$, w.p.~1, and $\lim_{n\rightarrow\infty}\hat{\theta}_n=\theta_*$,
w.p.~1. 
\end{cor}



\section{Asymptotics under Model-Correctness}
\label{sec:asyL}
To examine the asymptotic behaviour of Information Criteria like AIC or
BIC one has to investigate the behaviour
of the log-likelihood evaluated at the MLE, $\ell_{\hat{\theta}_{n}}(y_{0:n-1})$, for increasing data size $n$.
Following closely \cite{douc:14}, we first pose the following 
assumption, with $\theta_{*}\in\Theta$
as determined in Proposition~\ref{prop:uniform} and Assumption \ref{ass:six}.
Here and in the sequel, all gradients and Hessians -- represented by $\nabla$ and $\nabla \nabla^{\top}$ respectively, adopting an `applied mathematics' notation -- are  
w.r.t.~the model parameter(s) $\theta$.

\begin{assumption}
\label{ass:diff}
 $\theta_*$ is in the interior of $\Theta$, and there exists  $\epsilon>0$ and an open neighbourhood $\mathcal{B}_{\epsilon}(\theta_*):=\{\theta\in\Theta:|\theta-\theta_{*}|<\epsilon\}$
of $\theta_{*}$ such that the following hold.
\begin{itemize}
\item[(i)] For any $(x,x{'})\in\mathsf{X}\times\mathsf{X}$
and $y\in\mathsf{Y},$ $\theta\mapsto q_{\theta}(x{'}|x)$ and
$\theta\mapsto g_{\theta}(y|x)$ are twice continuously differentiable on $\mathcal{B}_{\epsilon}(\theta_*)$.
\item[(ii)] $\sup_{\theta\in\mathcal{B}_{\epsilon}(\theta_*)}\sup_{x,x{'}\in\mathsf{X}^2}\big\{\, \big| \nabla\log q_{\theta}(x{'}|x)\big| + \big|\nabla\nabla^{\top}\log q_{\theta}(x{'}|x)\big|\,\big\} <\infty$.
\item[(iii)] For some $\delta>0$,
$$\mathbb{E}\,\big[\sup_{\theta\in\mathcal{B}_{\epsilon}(\theta_*)}\sup_{x\in\mathsf{X}}\big\{\,\big| \nabla\log g_{\theta}(y_0|x)\big|^{2+\delta} +\big| \nabla\nabla^{\top}\log g_{\theta}(y_0|x)\big|\,\big\}\, \big]<\infty.$$
%
%
%
\end{itemize}
\end{assumption}
\noindent 
$|\cdot|$ denotes the Euclidean norm for vector input or 
one of the standard equivalent matrix norms for matrix input. Assumption \ref{ass:diff} implies that, for any fixed $n$
the log-likelihood function is twice continuously differentiable in $\mathcal{B}_{\epsilon}(\theta_*)$ (standard use of bounded convergence theorem from Assumption \ref{ass:diff}(i)). Also, the score function has 
finite $(2+\delta)$--moment and the Hessian finite first moment, for any $\theta\in\mathcal{B}_{\epsilon}(\theta_*)$;  the proof of these statements requires use of Fisher's identity (used later on) together with parts (ii), (iii) of Assumption \ref{ass:diff} involving the gradient   
for the score function, and Louis' identity (see e.g.~\cite{poyi:11} for background on Fisher's, Louis' identities) for the Hessian together with the stated conditions 
for the matrices of second derivatives. We avoid further details.

We start off with a standard Taylor expansion,
\begin{align}
\ell_{\hat{\theta}_n}&(y_{0:n-1}) = \ell_{\theta_*}(y_{0:n-1}) +  \tfrac{\nabla^{\top} 
\ell_{\theta_*}(y_{0:n-1})}{\sqrt{n}}\sqrt{n}(\hat{\theta}_n - \theta_*)\nonumber \\[-0.1cm] & + \tfrac{1}{2}\sqrt{n}(\hat{\theta}_n - \theta_*)^{\top}\Big[\int_{0}^{1} \tfrac{\nabla \nabla^{\top} 
\ell_{s\hat{\theta}_n+(1-s)\theta_*}(y_{0:n-1})}{n}ds\,\Big]\sqrt{n}(\hat{\theta}_n - \theta_*),
\label{eq:11}
\end{align}
together with a corresponding one for the score function,
\begin{align}
0&\equiv\tfrac{\nabla \ell_{\hat{\theta}_n}(y_{0:n-1})}{\sqrt{n}} \nonumber \\ &\qquad \qquad =\tfrac{\nabla \ell_{\theta_*}(y_{0:n-1})}{\sqrt{n}} +  \Big[\int_{0}^{1}\tfrac{\nabla \nabla^{\top} 
\ell_{s\hat{\theta}_n+(1-s)\theta_*}(y_{0:n-1})ds}{n}\,\Big]\sqrt{n}(\hat{\theta}_n - \theta_*).
\label{eq:22}
\end{align}
%
%
%
We will look at the asymptotic properties of the 
score function terms and the integral involving the Hessian, i.e.~of, 
\begin{equation}
\label{eq:twoterms}
\nabla
\ell_{\theta_*}(y_{0:n-1})/\sqrt{n}, \quad \int_{0}^{1}\tfrac{\nabla \nabla^{\top} 
\ell_{s\hat{\theta}_n+(1-s)\theta_*}(y_{0:n-1})ds}{n},
\end{equation}
starting from the former.

We will sometimes work under the assumption of model-correctness
and we shall be clear when that is the case.
\begin{assumption}
\label{ass:cor}
The dynamics of the data generating process $\{y_k\}_{k\ge 0}$ 
correspond to those of the HMM with initial distribution 
$x_0\sim \eta(\cdot)\equiv  \pi_{\theta_*}^{X}(\cdot)$, transition kernel $Q_{\theta_*}(\cdot|x)$ and observation kernel $G_{\theta_*}(\cdot|x)$. 
\end{assumption}
\noindent  For results that do not refer to Assumption \ref{ass:cor}, $\theta_*$ still makes sense as per its definition in Proposition~\ref{prop:uniform}.
Using Jensen's inequality, and for $\theta_*$ corresponding to the true parameter, 
one can easily check that $\ell(\theta)\le \ell(\theta_*)$, so indeed 
the true parameter coincides with~$\theta_*$ given in Proposition \ref{prop:uniform}.  
 
Following \cite[Ch.~13]{douc:14} 
we obtain:
\begin{itemize}
\item[1.] Re-write the score function evaluated at $\theta=\theta_*$ as 
\begin{equation}
\label{eq:split}
\nabla\ell_{\theta_*}(y_{0:n-1}) = \sum_{i=0}^{n-1}\big[\,\nabla \ell_{\theta_*}(y_{0:i})-\nabla\ell_{\theta_*}(y_{0:i-1}) \,\big],  
\end{equation} 
under the convention that $\nabla \ell_{\theta_*}(y_{0:-1})\equiv 0$.
The above differences will be shown to converge -- for increasing data size $n$, in an appropriate sense -- 
to stationary (and ergodic) martingale increments.
\item[2.] Using Fisher's identity, 
one has, for $y_{0:k}\in \mathsf{Y}^{k+1}$, $k\ge 0$,
\begin{align*}
\nabla\ell_{\theta_*}(y_{0:k})&=  
\int_{\mathsf{X}^{k+1}} 
\nabla\log p_{\theta_*}(x_{0:k},y_{0:k}) 
 \,p_{\theta_*}(x_{0:k}| y_{0:k})\mu^{\otimes (k+1)}(dx_{0:k}) \\[0.1cm]
& =
  \sum_{j=0}^{k} \int_{\mathsf{X}^2} d_{\theta_*}(x_{j-1},x_j,y_j) p_{\theta_*}(x_{j-1:j}|y_{0:k})
  \mu^{\otimes 2}(dx_{j-1:j}),  
\end{align*}
where we have defined 
\begin{align*}
d_{\theta_*}(x_{j-1},x_j,y_j):= \nabla \log\,[\,q_{\theta_*}(x_j|x_{j-1})\,g_{\theta_*}(y_j|x_j)\,], \quad j\ge 0,
\end{align*}
with the conventions
\begin{align*}
d_{\theta_*}(x_{-1},x_0,y_0)\equiv
d_{\theta_*}(x_0,y_0)
\equiv \nabla \log\,[\,\eta(x_0)\,g_{\theta_*}(y_0|x_0)\,]
\end{align*}
and the one
\begin{align*}
\int_{\mathsf{X}^2} d_{\theta_*}(x_{-1},x_0,y_0) &p_{\theta_*}(x_{-1:0}|y_{0:k})\mu^{\otimes 2}({dx_{-1:0}}) \\ 
&\equiv \int_{\mathsf{X}} d_{\theta_*}(x_0,y_0) p_{\theta_*}(x_{0}|y_{0:k})\mu(dx_0).
\end{align*}
Thus, we have, for $i\ge 0$,
\begin{align}
h_{\theta_*}(y_{0:i}):&=\nabla\ell_{\theta_*}(y_{0:i})-\nabla\ell_{\theta_*}(y_{0:i-1})
\label{eq:hdefine}  \\ &= 
\int_{\mathsf{X}^2} d_{\theta_*}(x_{i-1},x_i,y_i) p_{\theta_*}(x_{i-1:i}|y_{0:i})
\mu^{\otimes 2}(dx_{i-1:i}) \nonumber \\ &\quad\,\,\, + \sum_{j=0}^{i-1}\Big[\, \int_{\mathsf{X}^2} d_{\theta_*}(x_{j-1},x_j,y_j) p_{\theta_*}(x_{j-1:j}|y_{0:i})
\mu^{\otimes 2}(dx_{j-1:j}) \nonumber \\[-0.1cm] & \quad \quad \quad - \int_{\mathsf{X}^2} d_{\theta_*}(x_{j-1},x_j,y_j) p_{\theta_*}(x_{j-1:j}|y_{0:i-1})\mu^{\otimes 2}(dx_{j-1:j})  \,\Big]. \nonumber
\end{align}
\item[3.] To obtain stationary increments for increasing $n$, \cite{douc:14} work with (for $i\ge 0$)
\begin{align*}
h_{\theta_*}&(y_{-\infty:i}): = 
\int_{\mathsf{X}^2} d_{\theta_*}(x_{i-1},x_i,y_i) p_{\theta_*}(x_{i-1:i}|y_{-\infty:i})
\mu^{\otimes 2}(dx_{i-1:i}) \\ &\,\, + \sum_{j=-\infty}^{i-1}\Big[\, \int_{\mathsf{X}^2} d_{\theta_*}(x_{j-1},x_j,y_j) p_{\theta_*}(x_{j-1:j}|y_{-\infty:i})\mu^{\otimes 2}(dx_{j-1:j}) \\ & \qquad \qquad - \int_{\mathsf{X}^2} d_{\theta_*}(x_{j-1},x_j,y_j) p_{\theta_*}(x_{j-1:j}|y_{-\infty:i-1})\mu^{\otimes 2}(dx_{j-1:j})  \,\Big].
\end{align*} 
Following \cite[Proposition 13.20]{douc:14}, integrals involving infinitely long data sequences 
 of the form 
\begin{equation*} 
 \int_{\mathsf{X}^2} d_{\theta_*}(x_{j-1},x_j,y_j) p_{\theta_*}(x_{j-1:j}|y_{-\infty:i})\mu^{\otimes 2}(dx_{j-1:j}), \quad j\le i, \quad i \ge  0, 
\end{equation*} 
 appearing above are defined as a.s.~or $L_2$-limits of the random variables $\int_{\mathsf{X}^2} d_{\theta_*}(x_{j-1},x_j,y_j) p_{\theta_*}(x_{j-1:j}|y_{-m:i})\mu^{\otimes 2}(dx_{j-1:j})$, 
with $m\rightarrow \infty$, given Assumptions \ref{ass:1}-\ref{ass:diff}.
A small modification of the derivations in 
\cite[Ch.13]{douc:14} (they look at second moments) gives that, 
under Assumptions~\ref{ass:1}-\ref{ass:diff}, and for constant $\delta>0$ as defined in Assumption \ref{ass:diff}(iii),
for $i \ge 0$,
\begin{align}
\big\| h_{\theta_*}(y_{0:i} ) -& h_{\theta_*}(y_{-\infty:i}) \big\|_{2+\delta} \nonumber \\
&\le 12 \,\mathbb{E}^{1/(2+\delta)}\big[\!\!\sup_{x,x'\in\mathsf{X}}|d_{\theta_*}(x,x',y_0)|^{2+\delta}\,\big]      \,\tfrac{\rho^{i/2-1}}{1-\rho}, \label{eq:hbound}
\end{align}
where $\rho = 1 - \sigma^-/\sigma^+$.
(The expectation in the upper bound is finite due to 
Assumption \ref{ass:diff}(ii),(iii).)  
Here and below, $\|\cdot\|_a$, $a\ge 1$, denotes the $L_a$--norm of the variable under consideration.
From triangular inequality we have,
\begin{align*}
\Big\|  n^{-1/2} \sum_{i=0}^{n-1} \big\{ h_{\theta_*}&(y_{0:i} )-  h_{\theta_*}(y_{-\infty:i})  \big\}\Big\|_{2+\delta} \\ &\le n^{-1/2} 
\sum_{i=0}^{n-1}\| h_{\theta_*}(y_{0:i} ) - h_{\theta_*}(y_{-\infty:i}) \|_{2+\delta}.
\end{align*}
%
Thus, recalling equation (\ref{eq:split}) and definition (\ref{eq:hdefine}), 
the bound (\ref{eq:hbound}) implies
\begin{align}
\label{eq:nice}
\tfrac{\nabla \ell_{\theta_*}(y_{0:n-1})}{\sqrt{n}} =
n^{-1/2} \sum_{i=0}^{n-1}  h_{\theta_*}(y_{-\infty:i})  
+ \mathcal{O}_{L_{2+\delta}}(n^{-1/2}).
\end{align}
For $a\ge 1$ and a sequence of positive reals $\{b_k\}$, $\mathcal{O}_{L_{a}}(b_n)$ denotes a sequence of random variables 
with $L_{a}$-norm being $\mathcal{O}(b_n)$.
\item[4.] At this point we are required  to make explicit use of the model-correctness 
Assumption~\ref{ass:cor}. We have
\begin{align*}
&\EE\,[\,h_{\theta_*}(y_{-\infty:i})|y_{-\infty:i-1}\,]  = 
\EE\,\Big[\, \EE\,[\,d_{\theta_*}(x_{i-1},x_i,y_i)|y_{-\infty:i}\,]\,\Big|\, y_{-\infty:i-1}\,\Big]
 \\ & + \sum_{j=-\infty}^{i-1}\EE\,\Big[\,
\big\{\,\EE\,[\, d_{\theta_*}(x_{j-1},x_j,y_j) |y_{-\infty:i}]  \\  
&\qquad\qquad\qquad\qquad\qquad - \EE\,[\,d_{\theta_*}(x_{j-1},x_j,y_j)|y_{-\infty:i-1}]\,\big\} \,\Big|\, y_{-\infty:i-1}\,\Big]
\end{align*}
Each term in the sum is trivially $0$.
For the first term, we have, 
\begin{align*}
 \EE\,[\,d_{\theta_*}(x_{i-1},x_i,y_i)|y_{-\infty:i-1}\,]
  &=   \EE\,\big[\,\EE\,[\,d_{\theta_*}(x_{i-1},x_i,y_i)\,|\,x_{i-1},y_{-\infty:i-1}\,]\,\big] \\
& \equiv 0. 
\end{align*}
%
%
 %
Notice that we have indeed used the model correctness assumption 
to obtain the latter result. 
So, terms $h_{\theta_*}(y_{-\infty:i})$ make up  a strongly stationary, 
 ergodic (they inherit the properties of the data generating process) martingale increment sequence -- of finite second moment -- under the filtration generated by the data.
Using a CLT \citep{hall:80}
and the LIL of \cite{stou:70} for such sequences
allows for control over the martingales
\begin{equation*}
M_{n,j} := \sum_{i=0}^{n-1}  (h_{\theta_*}(y_{-\infty:i}))_j\ ,
\quad 1\le j \le d. 
\end{equation*}
Subscript $j$ indicates the $j$-th component of the $d$-dimensional vectors. 
In particular, we have the CLT (`$\Rightarrow$' denotes weak convergence, and $N_d(a,B)$ 
the $d$-dimensional Gaussian law with mean $a$ and covariance matrix $B$)
\begin{align}
\label{eq:maCLT}
\frac{M_{n}}{\sqrt{n}} \Rightarrow N_d(0,\mathcal{J}_{\theta_*}),
\end{align} 
where we have defined,
\begin{align}
\label{eq:lmatrix}
\mathcal{J}_{\theta_*} = \EE\,[\,h_{\theta_*}(y_{-\infty:0})h_{\theta_*}(y_{-\infty:0})^{\top} \,].
\end{align}
Also, we have the LIL \citep{stou:70}, 
\begin{align}
\label{eq:maLIL}
\limsup_{n} \tfrac{|M_{n,j}|}{ \sqrt{2n\log\log n}}  = \EE^{1/2}\big[\, (h_{\theta_*}(y_{-\infty:0}))_j^2\,\big], \quad 1\le j \le d, \quad w.p.\,1. 
\end{align}
\end{itemize}
We now turn to the second term in (\ref{eq:twoterms}). 
\begin{prop}
\label{prop:det}
Under Assumptions \ref{ass:1}-\ref{ass:diff}, we have that, w.p.~1, 
\begin{equation*}
\lim_{\delta\rightarrow 0}\lim_{n\rightarrow \infty} \sup_{\theta\in\mathcal{B}_\delta(\theta_*)} \big|  (-\nabla \nabla^{\top} \ell_{\theta}(y_{0:n-1})/n )   -  \mathcal{J}_{\theta_*}     \big| = 0.
\end{equation*}
\end{prop}
\begin{proof}
This is Theorem 13.24 of \cite{douc:14}.
\end{proof}

\begin{prop}
\label{prop:matrix}
Under Assumptions \ref{ass:1}-\ref{ass:diff} we have that, w.p.~1,
\begin{align*}
J_{\theta_*}(y_{0:n-1}):= -\int_{0}^{1}\tfrac{\nabla \nabla^{\top} 
\ell_{s\hat{\theta}_n+(1-s)\theta_*}(y_{0:n-1})ds}{n} \longrightarrow \mathcal{J}_{\theta_*}.
\end{align*}
\end{prop}
\begin{proof}
This is implied immediately from Proposition \ref{prop:det} (recall that 
$\hat{\theta}_n\rightarrow\theta_*$). 
\end{proof}

\noindent Notice that this result does not require the assumption of model correctness.
 We do make the following assumption on the HMM model under consideration.
\begin{assumption}
\label{ass:non}
The matrix $\mathcal{J}_{\theta_*}\in \mathbb{R}^{d\times d}$ is non-singular.
\end{assumption}

 We summarise the results in this part with a proposition and theorem.

\begin{prop}
\label{prop:all}
\begin{itemize}
\item[(i)] Under Assumptions \ref{ass:1}-\ref{ass:diff}, \ref{ass:non} we have,
for all large enough~$n$, 
\begin{align*}
\ell_{\hat{\theta}_n}(y_{0:n-1})  =
\ell_{\theta_*}(y_{0:n-1}) + 
\tfrac{1}{2}\tfrac{\nabla^{\top} 
\ell_{\theta_*}(y_{0:n-1})}{\sqrt{n}}J_{\theta_*}^{-1}(y_{0:n-1})
\tfrac{\nabla
\ell_{\theta_*}(y_{0:n-1})}{\sqrt{n}},
\end{align*}
where 
$J_{\theta_*}(y_{0:n-1})\rightarrow \mathcal{J}_{\theta_*}$, w.p.~1, for the non-singular matrix 
$\mathcal{J}_{\theta_*}$ defined  in (\ref{eq:lmatrix}).
\item[(ii)] 
 Under Assumptions \ref{ass:1}-\ref{ass:diff}, \ref{ass:non} we have, for all large enough~$n$,  
\begin{align*}
\tfrac{\nabla \ell_{\theta_*}(y_{0:n-1})}{\sqrt{n}} =
n^{-1/2} \sum_{i=0}^{n-1}  h_{\theta_*}(y_{-\infty:i})  
+ n^{-1/2} R_n
\end{align*}
where $\|h_{\theta_*}(y_{-\infty:i})\|_{2+\delta} +  \|R_n\|_{2+\delta} \le  C$, for  $\delta>0$ as determined in Assumption \ref{ass:diff}(iii) and a constant $C>0$.
\item[(iii)] Under Assumptions \ref{ass:1}-\ref{ass:non},
w.p.~1, as $n\rightarrow \infty$, $n^{-1/2}R_n\rightarrow 0$, and
we have the CLT 
\begin{gather*}
\frac{\nabla \ell_{\theta_*}(y_{0:n-1})}{ \sqrt{n}}  \Rightarrow
N_d(0,\mathcal{J}_{\theta_*}).
\end{gather*}
\end{itemize}
\end{prop}

\begin{proof}
The equation in part (i) is a combination of equations (\ref{eq:11}), (\ref{eq:22}), assuming that $n$ is big enough to permit inversion of the involved matrix.
(ii) is simply a rewriting of earlier calculations. 
The CLT in (iii) is trivial. 
\end{proof}

\begin{theorem}
\label{th:LIL}
Under Assumptions \ref{ass:1}-\ref{ass:non}, we have the LIL,
\begin{gather*}
\limsup_{n} \tfrac{  | (\nabla \ell_{\theta_*}(y_{0:n-1}))_j |}{ \sqrt{2n\log\log n}}  = \EE^{1/2}\big[\,   (h_{\theta_*}(y_{-\infty:0}))_j^2\,\big], \quad 1\le j \le d.
\end{gather*} 
\end{theorem}
\begin{proof}
From Proposition \ref{prop:all}(ii),  using Markov inequality, we have 
for $1\le j\le d$ and any $\epsilon>0$,
\begin{align*}
\mathbb{P}\,[\,| R_{n,j}|\geq\epsilon\, \sqrt{n}\,] = 
\mathbb{P}\,[\,| R_{n,j}|^{2+\delta}\geq\epsilon^{2+\delta}\, n^{1+\delta/2}\,]   & 
\leq\frac{\mathbb{E}\,| R_{n,j}|^{2+\delta}}{\epsilon^{2+\delta}\,n^{1+\delta/2}}. 
\end{align*}
Thus,
\begin{align*}
\sum_{n=0}^{\infty} \mathbb{P}\,[\,| R_{n,j}|\geq\epsilon\, \sqrt{n}\,] <\infty,
\end{align*}
and the Borel-Cantelli lemma gives that 
\begin{align*}
\mathbb{P}\,[\,|R_{n,j}|\geq\epsilon\, \sqrt{n},\,\,\textrm{infinitely often in } n\,] = 0.
\end{align*}
Equivalently, w.p.~1, $| R_{n,j}|/\sqrt{n}\rightarrow 0$. 
The proof is completed via the martingale LIL in (\ref{eq:maLIL}).
\end{proof}


\section{Model Selection Criteria for HMMs}
\label{sec:back}

We provide a brief illustration for the derivation of AIC and BIC, with focus on HMMs.
Results obtained that explicitly connect BIC and the evidence will allow for deriving
consistency properties for the evidence directly after studying the BIC criterion later in the paper.

\subsection{BIC and Evidence for HMMs}

We consider the derivation of BIC for a general
HMM.  BIC is used by~\cite{schw:78} and
 can be obtained by applying a Laplace approximation
at the  calculation of
the marginal likelihood (or evidence) of the model under consideration. Consideration of the sequence 
of log-likelihood functions over the data size $n$ (see e.g.~\cite{kass:90} for the concept of `Laplace-regular' models) provide 
analytical, sufficient conditions for controlling the difference between 
the evidence and BIC. 
%
We briefly review the Taylor expansions 
underlying the derivation of BIC and provide the regularity conditions 
that control its difference from the evidence in the context of HMMs.
Compared with \cite{kass:90}, weaker conditions are required here,
as BIC derives from an $\mathcal{O}(n^{-1})$ approximation, in an a.s.~sense, of the evidence (rather than $\mathcal{O}(n^{-2})$ expansions looked at in the Laplace-regular framework).

Let $\pi(\theta)d\theta$ be a prior for parameter $\theta$ -- for simplicity 
we assume that $d\theta$ is the Lebesgue measure on $\mathbb{R}^{d}$. 
The evidence is given by
\begin{align}
\label{eq:evidence}
p(y_{0:n-1}) & =\int_{\Theta}\pi(\theta)
\exp\big\{\ell_\theta(y_{0:n-1})\big\}d\theta.
\end{align}
%
%
%
We define 
\begin{equation*}
\mathsf{J}(y_{0:n-1}):= -\tfrac{\nabla \nabla^{\top} 
\ell_{\hat{\theta}_n}(y_{0:n-1})}{n}.
\end{equation*}
We will be explicit on regularity conditions in the statement of the 
proposition that follows. Following similar steps as in \cite{kass:90}, we apply a fourth-order Taylor expansion around the MLE $\hat{\theta}_{n}$
that gives -- for $u:=\sqrt{n}(\theta-\hat{\theta}_n)$,
\begin{align}
\ell_{\theta}(y_{0:n-1}) & =\ell_{\hat{\theta}_{n}}(y_{0:n-1})-\tfrac{1}{2}\,u^{\top}\mathsf{J}(y_{0:n-1})\,u
\nonumber \\ &\qquad \qquad +
\tfrac{1}{6}\,n^{-1/2} \sum_{i,j,k=1}^{d}
u_i u_j u_k \, \tfrac{\partial_{\theta_i}\partial_{\theta_j}\partial_{\theta_k}\ell_{\hat{\theta}_n}(y_{0:n-1})}{n}  + R_{1,n},
\label{eq:exps}
\end{align}
for residual term $R_{1,n}$ (in the integral form expansion)  involving fourth-order derivatives
of $\theta\mapsto\ell_{\theta}(y_{0:n-1})/n$ evaluated at $$\xi=a \hat{\theta}_n + (1-a)\theta,$$ for some $a\in[0,1]$, fourth order polynomials of $u$,
and a factor of $n^{-1}$, see e.g.~Ch.14 of \cite{lang:2012} for details on such expansions. Notice 
we have used $\nabla \ell_{\hat{\theta}_{n}}(y_{0:n-1})=0$.
For the prior density we have 
\begin{align*}
\pi(\theta)   = \pi(\hat{\theta}_{n})+n^{-1/2}\,\nabla^{\top} \pi(\hat{\theta}_{n})\,u  + R_{2,n},
\end{align*}
for the integral residual term  $R_{2,n}$ with second-order derivatives
of $\pi(\theta)$, second-order polynomial of $u$ and a factor of $n^{-1}$.
%
Using a second order expansion for $x\mapsto e^{x}$, 
only for the terms beyond the quadratic in $u$ in (\ref{eq:exps}), we get 
\begin{align}
\frac{p(y_{0:n-1})}{p_{\hat{\theta}_{n}}(y_{0:n-1})} =
&\int_{\Theta}e^{-\frac{1}{2}\,u^{\top}\mathsf{J}(y_{0:n-1})\,u} \times \nonumber \\[-0.1cm] &\qquad \quad \{\pi(\hat{\theta}_{n})+n^{-1/2}\,m(u,y_{0:n-1}) + R_n \}\,d\theta,
\label{eq:taylor}
\end{align}
where we have separated the term (later on removed as having zero mean under a Gaussian integrator) 
\begin{align*}
m(u,y_{0:n-1}) = \tfrac{1}{6}\,n^{-1/2} \sum_{i,j,k=1}^{d}
u_i u_j u_k \, \tfrac{\partial_{\theta_i}\partial_{\theta_j}\partial_{\theta_k}\ell_{\hat{\theta}_n}(y_{0:n-1})}{n} + \nabla^{\top} \pi(\hat{\theta}_{n})\,u;
\end{align*}
 the residual term $R_n$ can be deduced from the calculations. 
 
\begin{rem} 
\label{rem:Laplace}
 The Laplace-regular setting of \cite{kass:90} 
 provides concrete conditions for the above derivations to be valid and for controlling 
 the deduced residual terms. Apart from the standard assumptions on the existence 
 of derivatives and a bound on the fourth order derivatives of $\ell_{\theta}(y_{0:n-1})$ close to 
 $\theta_*$ -- the latter being defined in Proposition \ref{prop:uniform} as the limit of $\hat{\theta}_n$ -- the following are also required:
\begin{itemize}
\item[(i)] 
For any $\delta>0$, w.p.~1,
	\begin{align*}
	\limsup_{n}\sup_{\theta\in\Theta - \mathcal{B}_{\delta}(\theta_*)} \big\{\tfrac{1}{n}\,\big(\ell_{\theta}(y_{0:n-1})-\ell_{\theta_*}(y_{0:n-1})\big)\big\} <0;
	\end{align*}
	\item[(ii)] For some $\epsilon>0$, $\mathcal{B}_{\epsilon}(\theta^*)\subseteq \Theta$, and w.p.~1,
	\begin{align*}
\limsup_n\sup_{\theta\in\mathcal{B}_{\epsilon}(\theta_*)}
	\big\{\,\mathrm{det}\big(-\nabla \nabla^{\top}\ell_{\theta}(y_{0:n-1})/n\big)\big\} & >0.
	\end{align*}
\end{itemize}
 
\noindent Note  that (i) is implied by Proposition \ref{prop:uniform} and identifiability 
Assumption~\ref{ass:six}. Also, Proposition \ref{prop:det} and Assumption \ref{ass:non} imply (ii). 
\end{rem} 
 
\noindent 
Here, $\mathrm{det}(\cdot)$ denotes the determinant of a square matrix.
Following the above remark,  the Laplace-regular setting of \cite{kass:90} translates into 
the following assumption and proposition.
\begin{assumption}
\label{ass:Laplace}
\begin{itemize}
	\item[(i)] 
W.p.~1, $\theta\mapsto q_{\theta}(x'| x)$ and $\theta\mapsto g_{\theta}(y| x)$
are four-times continuously differentiable for $x,x'\in\mathsf{X}$,
$y\in\mathsf{Y}$; the prior $\theta\mapsto \pi(\theta)$
	 is two-times continuously differentiable.
		
\item[(ii)]  For some $\epsilon>0$, $\mathcal{B}_{\epsilon}(\theta_*)\subseteq \Theta$ and w.p.~1, for all $0\leq j_{1}\leq \cdots \leq  j_{k}\leq d$, $k\leq 4$
	\begin{align*}
	\limsup_{n}\sup_{\theta\in\mathcal{B}_{\epsilon}(\theta_*)}
	\big\{ 
	\tfrac{1}{n}\big|
	\partial_{\theta_{j_{1}}}\cdots\,\partial_{\theta_{j_{k}}}\ell_{\theta}(y_{0:n-1})
	\big|
	\big\}  <\infty.
	\end{align*}
	%
	%
			
	%
	%
	%
\end{itemize}
\end{assumption}

\begin{prop}
\label{prop:laplace}
Under Assumptions \ref{ass:1}-\ref{ass:diff}, \ref{ass:non}-\ref{ass:Laplace}, we have that, w.p.~1, 
\begin{align*}
\frac{p(y_{0:n-1})}{p_{\hat{\theta}_{n}}(y_{0:n-1})} = 
(2\pi)^{d/2}\,n^{-d/2}\,\{\mathrm{det}(\mathsf{J}(y_{0:n-1})\}^{-1/2}
\pi(\hat{\theta}_n)
\,(1  + \mathcal{O}(n^{-1})).
\end{align*}
\end{prop}
\begin{proof}
Under the assumptions, Theorem 2.1 of \cite{tadi:18} ensures that the the log-likelihood $\theta\mapsto\ell_{\theta}(y_{0:n-1})$ is four-times
continuously differentiable. Then, recall from Proposition \ref{prop:uniform} 
that $\theta\mapsto \ell_{\theta}(y_{0:n-1})/n$ converges uniformly to 
the continuous function $\theta\mapsto\ell(\theta)$ defined therein, which implies that $\hat{\theta}_n\rightarrow \theta_*$,
 with $\theta_*$ the unique maximiser of $\ell(\cdot)$ (under Assumption \ref{ass:six}) -- all these statements hold w.p.~1.
We choose sufficiently small $\delta>0$ (in Remark \ref{rem:Laplace}(i)), then $\epsilon=\epsilon_1$ and 
$\epsilon=\epsilon_2$ in Assumption \ref{ass:Laplace}(ii)
and Remark \ref{rem:Laplace}(ii) respectively, 
and  $\gamma>0$ such that for  
large enough $n$, $\mathcal{B}_{\delta}(\theta_*)\subseteq \mathcal{B}_{\gamma}(\hat{\theta}_n)\subseteq 
\mathcal{B}_{\min\{\epsilon_1,\epsilon_2\}}(\theta_*)
$.
We have that 
\begin{align*}
\frac{p(y_{0:n-1})}{p_{\hat{\theta}_{n}}(y_{0:n-1})} &= 
\int_{\Theta-\mathcal{B}_{\gamma}(\hat{\theta}_n)}\pi(\theta)\,
e^{n\times \frac{1}{n}\{\ell_\theta(y_{0:n-1})-\ell_{\hat{\theta}_n}(y_{0:n-1})\}}d\theta \\ & \qquad \qquad + \int_{\mathcal{B}_{\gamma}(\hat{\theta}_n)}\pi(\theta)\,
e^{\ell_\theta(y_{0:n-1})-\ell_{\hat{\theta}_n}(y_{0:n-1})}d\theta\\
&\leq e^{-c\,n} + \int_{\mathcal{B}_{\gamma}(\hat{\theta}_n)}\pi(\theta)\,
e^{\ell_\theta(y_{0:n-1})-\ell_{\hat{\theta}_n}(y_{0:n-1})}d\theta,
\end{align*}	 
for some $c>0$,
where we used Remark \ref{rem:Laplace}(i) to obtain the inequality.
It remains to treat the integral on $\mathcal{B}_{\gamma}(\hat{\theta}_n)$.
Applying the Taylor expansions as described in the main text
and continuing from (\ref{eq:taylor}) -- with the domain of integration 
now being $\mathcal{B}_{\gamma}(\hat{\theta}_n)$ -- will give,
\begin{align}
&\mathcal{I}_n:=\int_{\mathcal{B}_{\gamma}(\hat{\theta}_n)}\pi(\theta)\,
e^{\ell_\theta(y_{0:n-1})-\ell_{\hat{\theta}_n}(y_{0:n-1})}d\theta \nonumber
\\ &= 
 \int_{\mathcal{B}_{\gamma}(\hat{\theta}_n)}e^{-\frac{1}{2}\,u^{\top}\mathsf{J}(y_{0:n-1})\,u} \{\pi(\hat{\theta}_{n})+n^{-1/2}\,m(u,y_{0:n-1}) + R_n \}\,du.
 \label{eq:expa}
\end{align}
A careful, but otherwise straightforward,  consideration of the structure of the residual $R_n$ gives that,
under Remark \ref{rem:Laplace}(ii) and Assumption \ref{ass:Laplace}(ii), 
\begin{align*}
& \tfrac{1}{(2\pi)^{d/2}\{\mathrm{det}(\mathsf{J}(y_{0:n-1}))\}^{-1/2}}\int_{\mathcal{B}_{\gamma}(\hat{\theta}_n)}e^{-\frac{1}{2}\,u^{\top}\mathsf{J}(y_{0:n-1})\,u}\,|R_n|\,d\theta  = \mathcal{O}(n^{-1}).
\end{align*}
Thus, continuing from (\ref{eq:expa}), the change 
of variables $u=\sqrt{n}(\theta-\hat{\theta}_n)$ implies that, 
for $f(\cdot;\Omega)$ denoting the pdf 
of a centred $d$-dimensional Gaussian distribution with precision matrix $\Omega$,
\begin{align*}
& \mathcal{I}_n  = 
(2\pi)^{d/2}\,n^{-d/2}\,\{\mathrm{det}(\mathsf{J}(y_{0:n-1})\}^{-1/2} \\ &\qquad \qquad \times \int_{\mathcal{B}_{\gamma\sqrt{n}}(0)}
f(u;\mathsf{J}(y_{0:n-1})) \{\pi(\hat{\theta}_{n})+n^{-1/2}\,m(u,y_{0:n-1}) \}\,du \\  &\qquad \qquad \qquad \qquad \times(1+ \mathcal{O}(n^{-1}))
\end{align*}
The final result follows from the fact, that using Assumption \ref{ass:Laplace}(i), the integral appearing above is $\mathcal{O}(e^{-c'n})$ apart from 
the same integral over the whole $\mathbb{R}^{d}$, for some 
constant $c'>0$.
\end{proof}
Proposition \ref{prop:laplace} implies that, w.p.~1, 
\begin{align*}
\log p(y_{0:n-1})  = \ell_{\hat{\theta}_n}(y_{0:n-1}) -  
\tfrac{d}{2}\,\log n + \mathcal{O}(1) + \mathcal{O}(n^{-1}). 
\end{align*}
%
%
%
%
Ignoring the terms which are $\mathcal{O}(1)$ w.r.t.~$n$, we obtain that 
\begin{equation*}
2\log p(y_{0:n-1})\approx 2\ell_{\hat{\theta}_{n}}(y_{0:n-1})-d\log n.
\end{equation*}
%
%
Thus, working with the Laplace approximation to the evidence, one can derive the
standard formulation of the BIC,
\begin{align}
\mathrm{BIC} =-2\ell_{\hat{\theta}_{n}}(y_{0:n-1})+d\log n.\label{eq:BIC}
\end{align}
%


%
\begin{rem}
The above results 
provide an interesting conceptual reassurance.
Admitting the evidence as the core principle under which model comparison is carried out,  
if amongst a family of parametric HMM models, w.p.~1 one has the largest evidence for any big enough $n$, 
then BIC is guaranteed to eventually select that model as the optimal one.
\end{rem}

\begin{rem}
There is considerable work in the literature regarding consistency properties of the evidence (or Bayes Factor) for classes of models beyond the i.i.d.~setting, see e.g.~\cite{chat:18} and the references therein. In our approach, we have brought together results in the literature to deliver assumptions 
that -- whilst being fairly general -- were produced with HMMs in mind (and the connection between AIC and the evidence) and are relatively straightforward to be verified, indeed, for HMMs. Alternative approaches typically provide higher 
level conditions (see e.g.~above reference) in an attempt to preserve generality.  
\end{rem}
%
%

\subsection{AIC for HMMs}
AIC is developed in~\cite{akai:74} with its derivation discussed for i.i.d.~data  and Gaussian models of ARMA type.
Following more recent expositions (see e.g.~\cite{clae:08}),
AIC is based on the use of the Kullback-Leibler (KL) divergence 
for quantifying the distance between the true data-generating 
distribution and the probability model; an effort to reduce the bias of a `naive' estimator of the KL divergence leads to the formula for AIC. The case that one does not assume that the parametric model contains the true data distribution 
corresponds to a generalised version of AIC often called the Takeuchi Information Criterion (TIC), first proposed in~\cite{take:76}.
The above ideas are easy to be demonstrated in   simple settings (e.g.~\cite{clae:08} consider i.i.d.~and linear regression models). 

The framework connecting KL with AIC, in  the context of HMMs, can be developed as follows.
%
%
Let $v(dz_{0:n-1})$ denote the true data-generating distribution, $n\ge 1$.
A model is suggested in the form of a family of distributions $\{p_\theta(dz_{0:n-1});\,\theta\in\Theta \}$.
We assume that $v(dz_{0:n-1})$, $p_\theta(dz_{0:n-1})$ admit densities $v(z_{0:n-1})$, $p_\theta(z_{0:n-1})$ w.r.t.~$\nu^{\otimes n}$, $n\ge 1$.
We work with  the KL distance,
\begin{align}
\mathrm{KL}_n(\theta):&= \tfrac{1}{n} \int v(dz_{0:n-1})\log\frac{v(z_{0:n-1})}{p_{{\theta}}
	(z_{0:n-1})}  
\nonumber  \\
& = \tfrac{1}{n}   \int v(dz_{0:n-1})\log v(z_{0:n-1})-  \tfrac{1}{n}\int v(dz_{0:n-1})\log p_{{\theta}}(z_{0:n-1}). \label{eq:KL}
\end{align}
Therefore, minimising the above  discrepancy is equivalent to maximising
\begin{equation*}
\mathcal{R}_{n}(\theta) := \tfrac{1}{n}\int v(dz_{0:n-1})\log p_{{\theta}}(z_{0:n-1}).
\end{equation*}
Following standard ideas from cases models (e.g.~i.i.d.~models), one 
is interested in the quantity $\mathcal{R}_{n}(\hat{\theta}_n)$, but, in practice, 
has access only to the naive 
estimator $\frac{1}{n}\ell_{\hat{\theta}_{n}}(y_{0:n-1})$, the latter tending 
to have positive bias versus $\mathcal{R}_{n}(\hat{\theta}_n)$ due to the use of 
both the data and the data-induced MLE in its expression. AIC is then 
derived by finding the larger order term (of size $\mathcal{O}(1/n)$) in the discrepancy of the expectation and appropriately adjusting the naive estimator.

%
%
\begin{assumption}
\label{ass:AIC}
\begin{itemize}
\item[(i)] There exists a constant $C>0$, such that w.p.~1, 
\begin{equation*}
\sup_{n\ge 1} \sup_{\theta\in\Theta} \Big\{  \tfrac{1}{n} \big|\nabla \nabla^{\top} \ell_{\theta}(y_{0:n-1})    \big|   \Big\}   < C  .
\end{equation*}
\item[(ii)] There is some $n_0\ge 1$ such that w.p.~1, matrix $J_{\theta_*}^{-1}(y_{0:n-1})$  
-- defined in Proposition \ref{prop:matrix} --
is well-posed for all $n\ge n_0$, and there is a constant $C'>0$, such that w.p.~1,
\begin{equation*}
\sup_{n\ge n_0}|J_{\theta_*}^{-1}(y_{0:n-1})| < C'.
\end{equation*}
\end{itemize}
\end{assumption}
\noindent These are high-level assumptions -- especially Assumption \ref{ass:AIC}(ii) -- and a more analytical study is required for them to be of immediate practical use (or weakening them); but such a study would considerably deviate from the main purposes of this work. Our contribution is contained in the following proposition.%
\begin{prop}
\label{prop:AIC}Under Assumptions \ref{ass:1}-\ref{ass:non}, \ref{ass:AIC},
we have that 
\begin{align*}
\mathbb{E}\,\big[\, \tfrac{1}{n}\ell_{\hat{\theta}_{n}}(y_{0:n-1}) -  \mathcal{R}_{n}(\hat{\theta}_n)\,\big] 
 =  \tfrac{d}{n}  + o(n^{-1}). 
\end{align*}
\end{prop}
\begin{proof}
Use of a second-order Taylor expansion gives,
\begin{align}
&\tfrac{1}{n}\ell_{\hat{\theta}_{n}}(y_{0:n-1}) -  \mathcal{R}_{n}(\hat{\theta}_n) \nonumber \\ 
&=   \tfrac{1}{n}\ell_{\theta_{*}}(y_{0:n-1}) -  
\tfrac{1}{n}\int \ell_{{\theta}_*}(z_{0:n-1})v(dz_{0:n-1})
 \nonumber \\
& \quad   + \tfrac{1}{n}\tfrac{\nabla^{\top}\ell_{\theta_{*}}(y_{0:n-1})}{\sqrt{n}}\sqrt{n}(\hat{\theta}_{n}-\theta_{*}) -  \tfrac{1}{n}\Big\{\int v(dz_{0:n-1})\nabla^{\top}\ell_{\theta_{*}}(z_{0:n-1})\Big\}(\hat{\theta}_{n}-\theta_{*})  \nonumber\\
 & \quad +\tfrac{1}{2n}
\sqrt{n}(\hat{\theta}_{n}-\theta_{*})^{\top} 
\Big\{    
\int \mathcal{E}_{\theta_*}(y_{0:n-1},z_{0:n-1}) v(dz_{0:n-1}) 
\Big\}
\sqrt{n} (\hat{\theta}_{n}-\theta_{*}). \label{eq:2T}
\end{align}
where we have set
\begin{equation*}
\mathcal{E}_{\theta_*}(y_{0:n-1},z_{0:n-1}):= \int_{0}^{1}\tfrac{\nabla \nabla^{\top} 
\ell_{s\hat{\theta}_n+(1-s)\theta_*}(y_{0:n-1}) -\nabla \nabla^{\top} 
\ell_{s\hat{\theta}_n+(1-s)\theta_*}(z_{0:n-1})}{n}ds.
\end{equation*}
Taking expectations in (\ref{eq:2T}), notice that: i) the expectation of the first difference on
the right-hand-side is trivially $0$; ii) the integral appearing in the second difference is 
identically zero, since we are working under the correct model Assumption \ref{ass:cor}.
It remains to consider the expectation of the terms, 
\begin{gather}
\zeta_{n}:=\tfrac{1}{n}\tfrac{\nabla^{\top}\ell_{\theta_{*}}(y_{0:n-1})}{\sqrt{n}}\sqrt{n}(\hat{\theta}_{n}-\theta_{*});\nonumber  \\
\zeta_{n}':=\tfrac{1}{2n}
\sqrt{n}(\hat{\theta}_{n}-\theta_{*})^{\top} 
\Big\{    
\int \mathcal{E}_{\theta_*}(y_{0:n-1},z_{0:n-1}) v(dz_{0:n-1}) 
\Big\}
\sqrt{n} (\hat{\theta}_{n}-\theta_{*}). \label{eq:ante1}
\end{gather}
The first term rewrites as, using (\ref{eq:22}),
\begin{align*}
\zeta_n = \tfrac{1}{n}\times \tfrac{\nabla^{\top}\ell_{\theta_{*}}(y_{0:n-1})}{\sqrt{n}}
J_{\theta_*}^{-1}(y_{0:n-1})
 \tfrac{\nabla\ell_{\theta_{*}}(y_{0:n-1})}{\sqrt{n}}
\end{align*}
Thus, Proposition \ref{prop:all} gives that,
\begin{align*}
n \zeta_n \Rightarrow Z^{\top} \mathcal{J}^{-1}_{\theta_*}Z; \quad Z\sim N(0,\mathcal{J}_{\theta_*}).
\end{align*}
For weak convergence to imply convergence in expectation, we require uniform integrability.
Assumption \ref{ass:AIC}(ii) takes care of the difficult term $J_{\theta_*}^{-1}(y_{0:n-1})$. 
Then, Proposition \ref{prop:all}(iii) and the Marcinkiewicz--Zygmund inequality applied for martingales \citep{ibra:99}, give that 
\begin{align*}
\sup_{n} \| \tfrac{\nabla\ell_{\theta_{*}}(y_{0:n-1})}{\sqrt{n}} \|_2 < \infty.
\end{align*}
Thus, from Cauchy--Schwarz, we have 
\begin{align*}
\sup \| n \zeta_n \|_2 < \infty,  
\end{align*}
which implies uniform integrability for $\{n \zeta_n\}_{n}$. 
So, we have shown that,
\begin{equation}
\label{eq:AIC11}
\mathbb{E}\,[\,n\zeta_n\,]\,\rightarrow \mathbb{E}\,[\,Z^{\top} \mathcal{J}^{-1}_{\theta_*}Z\,] \equiv   d.
\end{equation}

We proceed  to term $\zeta_n'$ in (\ref{eq:ante1}).
%
Using again (\ref{eq:22}) and setting 
\begin{equation*}
A_{\theta_*}(y_{0:n-1}):= \nabla^{\top}\ell_{\theta_{*}}(y_{0:n-1})/\sqrt{n}\cdot J_{\theta_*}^{-1}(y_{0:n-1}), 
\end{equation*}
we have that,
\begin{align*}
2n \zeta'_n = 
 A_{\theta_*}(y_{0:n-1})
\Big\{\int \mathcal{E}_{\theta_*}(y_{0:n-1},z_{0:n-1})v(dz_{0:n-1})\Big\}  A_{\theta_*}^{\top}(y_{0:n-1}).
\end{align*}
Clearly, we can write,
\begin{align*}
\mathbb{E}\,&[\,2n \zeta'_n\,]  \\ &=  \int   \big\{A_{\theta_*}(y_{0:n-1})  \mathcal{E}_{\theta_*}(y_{0:n-1},z_{0:n-1})       A_{\theta_*}^{\top}(y_{0:n-1})\big\} (v\otimes v) (dy_{0:n-1},dz_{0:n-1}).
\end{align*}
From Proposition \ref{prop:det} we obtain that  $(v\otimes v) (dy_{0:n-1},dz_{0:n-1})$-a.s., 
we have that $\lim_n\mathcal{E}_{\theta_*}(y_{0:n-1},z_{0:n-1})\rightarrow \mathcal{J}_{\theta_*}-\mathcal{J}_{\theta_*} = 0$. 
This implies the weak convergence of $A_{\theta_*}(y_{0:n-1})  \mathcal{E}_{\theta_*}(y_{0:n-1},z_{0:n-1})       A_{\theta_*}^{\top}(y_{0:n-1})\Rightarrow 0$.
Assumption \ref{ass:AIC},  and arguments similar to the ones used for $\zeta_n$, 
imply uniform integrability for $\{n\zeta_n'\}$. We thus have $\mathbb{E}\,[\,n\zeta_n'\,]\rightarrow 0$.

This latter result together with (\ref{eq:AIC11}) complete the proof.
\end{proof}
\noindent Proposition \ref{prop:AIC} provides the underlying principle for use of the standard AIC,
\begin{align}
\mathrm{AIC}:= & -2\ell_{\hat{\theta}_{n}}(y_{0:n-1})+2d. \label{eq:-22}
\end{align}

\section{BIC, Evidence, AIC Consistency Properties}
\label{sec:hmm}

We will now use the results in Sections \ref{sec:asy}-\ref{sec:back} 
to examine the asymptotic properties of BIC, the evidence and AIC 
in the context of HMMs.
We define the notions of strong and weak consistency in model selection in a nested setting as
follows.

\begin{define}[Consistency of Model Selection Criterion]
\label{def:con}
Assume a sequence of nested parametric models 
\begin{equation*}
\mathcal{M}_{1}\subset \cdots \subset \mathcal{M}_{k}\subset \mathcal{M}_{k+1}\subset \cdots \subset \mathcal{M}_{p}, \quad 
\end{equation*}
for some fixed $p\ge 1$,
specified via a sequence of corresponding parameter spaces
$\Theta_1\subseteq \RR^{d_1}$, and $\Theta_{k+1}=\Theta_{k}\times \Delta \Theta_k$, 
$ \Delta \Theta_k\subseteq \RR^{d_{k+1}-d_k}$, $k\ge 1$, with $d_k < d_l$ for $k<l$.
Let $\mathcal{M}_{k^{*}}$, for some $k^*\ge 1$, be the smallest model containing the correct one -- the latter determined by the true parameter value $\theta_{*}^{k^*}(=:\theta_*)\in \Theta_{k^*}$.

Let $\mathcal{M}_{\hat{k}_{n}}$, for index $\hat{k}_{n}\ge 1$ 
based on data $\{y_0, \ldots, y_{n-1}\}\in \mathsf{Y}^{n}$, $n\ge 1$,
be the model selected via optimising
a Model Selection Criterion.
If it holds that $\lim_{n\rightarrow \infty}\hat{k}_{n}= k^{*}$, w.p.\,1,
then the Model Selection Criterion is called \emph{strongly consistent}. If it holds that $\lim_{n\rightarrow\infty}\hat{k}_{n}= k^{*}$, in probability, then the Model Selection Criterion
is called \emph{weakly consistent}.
\end{define}

\noindent We henceforth assume  that for each $1\le k\le p$,   $\mathcal{M}_k$ corresponds to a parametric HMM as defined 
in Section \ref{sec:asy}. The particular model under consideration will be implied by the corresponding parameter appearing in an expression or the superscript $k$ used in relevant functions; i.e., a quantity involving $\theta^{k}$ will refer to model $\mathcal{M}_{k}$.
E.g., $\theta_*^{k}\in \Theta^{k}$ is the a.s.~limit of the MLE, $\hat{\theta}_{n}^{k}$, for model $\mathcal{M}_{k}$, and such a limit has been shown to exist under Assumptions \ref{ass:1}-\ref{ass:six} for model $\mathcal{M}_{k}$.

\begin{assumption}
\label{ass:nested}
Assumptions \ref{ass:2}-\ref{ass:six} hold for each parametric 
model $\mathcal{M}_{k}$, for index $1\le k< k^*$; 
Assumptions \ref{ass:2}-\ref{ass:non} 
hold for each parametric 
model $\mathcal{M}_{k}$, for index $k^*\le k\le p$.
\end{assumption}
%
\begin{rem}
For a model $\mathcal{M}_{k}$ that contains $\mathcal{M}_{k^*}$ ($k>k^*$), for all of Assumptions \ref{ass:2}-\ref{ass:non} to hold, it is necessary  
that the parameterisation of the larger model $\mathcal{M}_{k}$ is such that non-identifiability issues 
are avoided. In a trivial example, for $\mathcal{M}_{k^*}$ corresponding to 
i.i.d.~data from $N(\theta_1,1)$, a larger model of the form $N(\theta_1,\exp(\theta_2))$
would satisfy Assumptions \ref{ass:2} and \ref{ass:non} (the main ones the relate to the shape, in the limit, of the log-likelihood and, consequently identifiability) -- one can check this -- whereas 
model $N(\theta_1+\theta_{2},1)$ would not. In practice, for a given application with nested models, one can most times easily deduce whether identifiability issues are taken care of, thus Assumptions \ref{ass:2}-\ref{ass:non} correspond to reasonable requirements over the larger models.  In general, only `atypical' parameterisations can produce non-identifyibility issues -- thus, also abnormal behavior of the log-likelihood function -- for the case of the larger model.
\end{rem}
\begin{prop}
\label{prop:diff}
Let $\lambda_{n}:=\ell_{\hat{\theta}_{n}^{k}}^{k}(y_{0:n-1})-\ell^{k^*}_{\hat{\theta}_n^{k^{*}}}(y_{0:n-1})$, for a $k\neq k^{*}$. 
Under Assumptions \ref{ass:2}-\ref{ass:non}, \ref{ass:nested} we have the following.
\begin{itemize}
\item[(i)]
If $\mathcal{M}_{k}\subset \mathcal{M}_{k^{*}}$, then $\lim_{n\rightarrow \infty} n^{-1}\lambda_{n}= \ell^{k}(\theta_{*}^{k})-\ell^{k^*}(\theta_{*})<0$, w.p.~1.
\item[(ii)] If $\mathcal{M}_{k}\supset \mathcal{M}_{k^{*}}$, then $\lambda_{n}\ge 0$ and  $\lambda_n = \mathcal{O}(\log\log n)$, w.p.~1.
\end{itemize}
\end{prop}
\begin{proof}
(i)\,\, From Proposition \ref{prop:uniform} we have, w.p.~1, 
\begin{align*}
n^{-1}&\big(\ell_{\hat{\theta}_{n}^{k}}(y_{0:n-1})
-\ell^{k^*}_{\hat{\theta}_n^{k^{*}}}(y_{0:n-1})\big)\rightarrow 
\ell^{k}(\theta_{*}^{k})-\ell^{k^*}(\theta_{*}) 
 \\ & \qquad \quad \equiv 
 \EE\,[\,\log p_{\theta_*^{k}}(y_0|y_{-\infty:-1})\,] - 
 \EE\,[\,\log p_{\theta_*}(y_0|y_{-\infty:-1})\,].
\end{align*}
Using Jensen's inequality and simple calculations, 
one obtains that 
\begin{align*}
 &\EE\,[\,\log p_{\theta_*^{k}}(y_0|y_{-\infty:-1})\,] - 
 \EE\,[\,\log p_{\theta_*}(y_0|y_{-\infty:-1})\,] \\
 & \qquad \le \log \int \tfrac{p_{\theta_*^{k}}(y_0|y_{-\infty:-1})}{p_{\theta_*}(y_0|y_{-\infty:-1})} p_{\theta_*}(y_{-\infty:0})dy_{-\infty:0}
 \equiv  \log1.
\end{align*}
For strict inequality, Assumptions \ref{ass:six} and \ref{ass:nested}  
imply that mapping $\theta\mapsto \ell^{k^*}(\theta)$ has the unique 
maximum $\theta_*\in\Theta_{k^*}$.
Thus, we cannot have $\ell^{k}(\theta_{*}^{k})=\ell^{k^*}(\theta_{*})$, as this would give (from the nested model structure) 
$\ell^{k^*}(\theta_{*})= \ell^{k^*}(\theta_{k},\theta_0)$ for some 
$\theta_0\in \prod_{l=k+1}^{k^*} \Delta \Theta_{l}$, 
with $(\theta_{k},\theta_0)^{\top}\neq \theta_{*}$ (otherwise 
the definition of correct model class would be violated).
 \\[0.2cm]
(ii)\,\, 
Having $\lambda_n\ge 0$ is a consequence of 
the log-likelihood for model $\mathcal{M}_k$ being maximised over a larger 
parameter domain than $\mathcal{M}_{k^*}$.
Then, notice that the limiting matrix $\mathcal{J}_{\theta_*}$  in Proposition \ref{prop:matrix} (for the notation used therein)
is positive-definite: it is non-negative-definite following its definition 
in (\ref{eq:lmatrix}); then, non-singularity Assumption \ref{ass:non} provides 
the positive-definiteness.
From Proposition \ref{prop:all}(i), the difference in the definition of 
$\lambda_n$ equals the difference of two quadratic forms, as the constants 
in the expression for the log-likelihood provided by Proposition \ref{prop:all}(i) are equal for models 
$\mathcal{M}_{k^*}$, $\mathcal{M}_{k}$ and cancel out.
As $\lambda_n\ge 0$, and both quadratic forms are non-negative, 
it suffices to consider the one for model $\mathcal{M}_k$. The a.s.~convergence of 
the positive-definite matrix in the quadratic form implies a.s.~convergence of 
its eigenvalues and eigenvectors. Thus, using Theorem \ref{th:LIL},
overall one has that $\lambda_n = 
\mathcal{O}(\sum_{i=1}^{d}\log\log n) = \mathcal{O}(\log\log n)$.
%
%
\end{proof}

\subsection{Asymptotic Properties of BIC and Evidence}

BIC is known to be strongly consistent 
in i.i.d.~settings and some particular non-i.i.d.~ones \citep{clae:08}. In the context of HMMs, \cite{gass:03} show strong consistency of BIC   for
observations that take a finite set of values.
The key tool to obtain strong consistency of BIC in
a general HMM is  LIL 
we obtained in Section \ref{sec:asyL}.  \cite{nish:88}
also uses LIL for the i.i.d.~setting to prove strong (and weak) consistency
of BIC. 

Recall that $k^*$ denotes the index of the correct model.


\begin{prop}
\label{prop:strong}
\begin{itemize}
\item[(i)]
Let $\hat{k}_{n}$  be the index of the selected model 
obtained via minimizing BIC as defined in (\ref{eq:BIC}).
Then, under Assumptions \ref{ass:1}-\ref{ass:non}, \ref{ass:nested},
we have that $\hat{k}_{n}\rightarrow k^{*}$, w.p.~1.
\item[(ii)] If $\hat{\mathsf{k}}_{n}$ denotes the index obtained via maximising the evidence in 
(\ref{eq:evidence}), then Assumptions \ref{ass:1}-\ref{ass:nested} imply that 
$\hat{\mathsf{k}}_{n}\rightarrow k^{*}$, w.p.~1.
\end{itemize}
\end{prop}
\begin{proof}
\begin{itemize}
\item[(i)]
We make use of Proposition \ref{prop:diff}. 
Indeed, in the case that $\mathcal{M}_{k}\subset \mathcal{M}_{k^{*}}$, Proposition \ref{prop:diff}(i) gives 
\begin{align*}
&\mathrm{BIC}_n(\mathcal{M}_{k})-\mathrm{BIC}_n(\mathcal{M}_{k^{*}})  \\ &\qquad =n\,\big\{ \tfrac{1}{n}\ell_{\hat{\theta}_{n}^{k^{*}}}(y_{0:n-1})-\tfrac{1}{n}\ell_{\hat{\theta}_{n}^{k}}(y_{0:n-1})-\tfrac{(d_k-d_{k^{*}})\log n}{n}\big\} ,\\
& \qquad \qquad \rightarrow+\infty\,\,\,\,w.p.\,1,
\end{align*}
therefore $\liminf_{n} \hat{k}_{n}\ge k^{*}$, w.p.\,1.
In the case
$\mathcal{M}_{k}\supset \mathcal{M}_{k^{*}}$, we obtain, 
from Proposition \ref{prop:diff}(ii),
\begin{align*}
\mathrm{BIC}_n &(\mathcal{M}_{k})-\mathrm{BIC}_n(\mathcal{M}_{k^{*}})  \\ &
=(d_k-d_{k^{*}})\log n-
\{ 
\ell_{\hat{\theta}_{n}^{k}}(y_{0:n-1})-
\ell_{\hat{\theta}_{n}^{k^{*}}}(y_{0:n-1})
\} \\
& =(d_k-d_{k^{*}})\log n-\mathcal{O}(\log\log n).
\end{align*}
Thus, w.p.~1, for all large enough $n$,  
$\mathrm{BIC}_n (\mathcal{M}_{k})-\mathrm{BIC}_n(\mathcal{M}_{k^{*}})>c_k>0$,
for some constant $c_k$.
%
%
Therefore, we have   $\limsup_{n} \hat{k}_{n}\le  k^{*}$, w.p.~1.
\item[(ii)] Given part (i), this now follows directly from Proposition \ref{prop:laplace}.
\end{itemize}
\end{proof}
\noindent  Therefore, BIC is strongly consistent for a general class of HMMs in the nested model setting we are considering here -- with a model assumed to be a correctly specified one.

\subsection{Asymptotic Properties of AIC}

We can be quite explicit about  the behaviour of AIC.
Consider the case $k>k^*$.
Making use of Proposition \ref{prop:all} we have 
\begin{align}
 \ell_{\hat{\theta}_{n}^{k}}(y_{0:n-1})-&\ell_{\hat{\theta}_{n}^{k^*}}(y_{0:n-1}) \nonumber \\  & = 
\tfrac{1}{2}\tfrac{\nabla^{\top} 
\ell_{\theta_*^k}(y_{0:n-1})}{\sqrt{n}}\mathcal{J}_{\theta_*^k}^{-1}
\tfrac{\nabla
\ell_{\theta_*^k}(y_{0:n-1})}{\sqrt{n}} \nonumber \\
&\qquad \qquad - 
\tfrac{1}{2}\tfrac{\nabla^{\top} 
\ell_{\theta_*}(y_{0:n-1})}{\sqrt{n}}\mathcal{J}_{\theta_*}^{-1}
\tfrac{\nabla
\ell_{\theta_*}(y_{0:n-1})}{\sqrt{n}} + \epsilon_n,
\label{eq:end0}
\end{align}
where $\epsilon_n = o(\log\log n)$, w.p.~1. 
Due to working with nested models, we have (immediately from the definition of $\mathcal{J}_{\theta_*^k}$, $\mathcal{J}_{\theta_*}$) 
\begin{align*}
\mathcal{J}:= \mathcal{J}_{\theta_*^k} = \left( \begin{array}{cc} \mathcal{J}_{11} & \mathcal{J}_{12} \\   
\mathcal{J}_{21} & \mathcal{J}_{22} 
\end{array} \right)\in \mathbb{R}^{d_k\times d_k}.
\end{align*}
where $\mathcal{J}_{11}\equiv \mathcal{J}_{\theta_*}$, and $\mathcal{J}_{12}$, $\mathcal{J}_{21}=\mathcal{J}_{12}^{\top}$ as deduced from $\mathcal{J}_{\theta_*^k}$. Similarly, the quantity $\nabla 
\ell_{\theta_*}(y_{0:n-1})$ forms the upper $d_{k^*}$-dimensional 
part of $\nabla
\ell_{\theta_*^k}(y_{0:n-1})$.
We will make use of the matrix equations implied by
\begin{align*}
\mathcal{J}&\mathcal{J}^{-1}=I_{d_k}\,\, \Longleftrightarrow \\ &\left( \begin{array}{cc} \mathcal{J}_{11} & \mathcal{J}_{12} \\   
\mathcal{J}_{21} & \mathcal{J}_{22} 
\end{array} \right)\left( \begin{array}{cc} (\mathcal{J}^{-1})_{11} & (\mathcal{J}^{-1})_{12} \\   
(\mathcal{J}^{-1})_{21} & (\mathcal{J}^{-1})_{22} 
\end{array} \right) = 
\left( \begin{array}{cc} I_{d_{k^*}} & 0_{d_{k^*}\times (d_k-d_{k^*})} \\   
0_{(d_k-d_{k^*})\times d_{k^*}} & I_{(d_k-d_{k^*})}
\end{array} \right).
\end{align*}
Given the above nesting considerations, some cumbersome but otherwise straightforward calculations give
\begin{align}
&\tfrac{\nabla^{\top} 
\ell_{\theta_*^k}(y_{0:n-1})}{\sqrt{n}}\,\mathcal{J}_{\theta_*^k}^{-1}\,
\tfrac{\nabla
\ell_{\theta_*^k}(y_{0:n-1})}{\sqrt{n}}
 - 
\tfrac{\nabla^{\top} 
\ell_{\theta_*}(y_{0:n-1})}{\sqrt{n}}\,\mathcal{J}_{\theta_*}^{-1}\,
\tfrac{\nabla
\ell_{\theta_*}(y_{0:n-1})}{\sqrt{n}} \nonumber\\
&\qquad \qquad \equiv  
\tfrac{\{ \,M\nabla
\ell_{\theta_*^k}(y_{0:n-1})\}^{\top}}{\sqrt{n}}\,D\, 
\tfrac{M\nabla
\ell_{\theta_*^k}(y_{0:n-1})}{\sqrt{n}},
\label{eq:end}
\end{align}
where we have set
\begin{align*}
M &:= \big(\, (\mathcal{J}^{-1})_{21}\,\,\,\,\{(\mathcal{J}^{-1})_{22}\}^{-1}\, \big)\in 
\mathbb{R}^{(d_k-d_{k^*})\times d_k}; \\[0.2cm]
D &:= \{(\mathcal{J}^{-1})_{22}\}^{-1}\in\mathbb{R}^{(d_k-d_{k^*})\times (d_k-d_{k^*})} .
\end{align*}
Consider the standard decomposition of the symmetric positive-definite $D$,
\begin{align*}
D = P \Lambda P^{\top}, 
\end{align*}
for orthonormal $P\in \mathbb{R}^{(d_k-d_{k}^*)\times (d_k-d_{k}^*)}$ and diagonal $\Lambda\in \mathbb{R}^{(d_k-d_{k}^*)\times (d_k-d_{k}^*)}$.
\begin{assumption}
\label{ass:nontrivial}
Define the martingale increments in $\mathbb{R}^{d_k-d_{k^*}}$,
$k>k^*$.
\begin{align*}
\tilde{h}_{\theta_*^k}(y_{-\infty:0}) := (\sqrt{\Lambda}P^{\top}M)\,h_{\theta_*^k}(y_{-\infty:0}).
\end{align*}
We have that $\mathbb{E}\,\big|\tilde{h}_{\theta_*^k}(y_{-\infty:0})\big|^2>0$.
\end{assumption}

\begin{prop}
\label{prop:clear}
Under Assumptions \ref{ass:1}-\ref{ass:non}, \ref{ass:nested}-\ref{ass:nontrivial}, we have 
that, for $k>k^*$,
\begin{align*}
\mathbb{P}\,\big[\,\mathrm{AIC}_n(\mathcal{M}_{k}) < \mathrm{AIC}_n(\mathcal{M}_{k^*}),\,\,\textrm{infinitely often in}\,\,n\ge 1\,\big] = 1.
\end{align*}
\end{prop}
\begin{proof}
Continuing from (\ref{eq:end0}), (\ref{eq:end}), the use of LIL for martingale increments 
will give that, w.p.~1,
\begin{align*}
 \limsup_n &\frac{\sqrt{2}\big\{\ell_{\hat{\theta}_{n}^{k}}(y_{0:n-1})- \ell_{\hat{\theta}_{n}^{k^*}}(y_{0:n-1})\}}{\log \log n}
 \\ &\qquad \qquad \ge \sup_{1\le j \le d_{k}-d_{k^*}} \mathbb{E}\,\big[\,(\tilde{h}_{\theta_*^k}(y_{-\infty:0}))^2_j\,\big] > 0.
\end{align*}
As the difference $\ell_{\hat{\theta}_{n}^{k}}(y_{0:n-1})- \ell_{\hat{\theta}_{n}^{k^*}}(y_{0:n-1})$ is of size 
$\Theta(\log\log n)-o(\log\log n)$ infinitely often, the result follows immediately. (The notation $\Theta(a_n)$ for  a positive sequence $\{a_n\}$ 
means that the sequence of interest is lower and upper bounded by $c a_n$ and $c' a_n$ respectively
for constants $0<c<c'$.)
\end{proof}

\noindent Thus, AIC is not a consistent Model Selection
Criterion -- in contrast with BIC. Still, it is well known
that AIC has desirable properties, e.g.~with regards to prediction error (in many cases the model chosen by AIC attains the
minimum maximum error in terms of prediction among models being
considered), or its 
minimax optimality. 
\citet{barr:99} is a comprehensive article of
this topic and shows that minimax optimality of AIC holds in many
cases, including the i.i.d., some non-linear models and for density
estimation. 
For works on the efficiency of 
AIC terms of prediction see \cite{shib:80, shib:81, shao1997asymptotic}. 
AIC is equivalent to LOO cross-validation \citep{stone1977asymptotic} for i.i.d.-type model structures. We 
refer to \citet{ding2017bridging, ding2018model} for a comprehensive review of AIC and
BIC. Note that \cite{yang2005can} shows that consistency of model selection and minimax optimality do not necessarily hold simultaneously.
Our main focus in this work is asymptotic behaviour of criteria from a model selection viewpoint, so we will not further examine the prediction perspective.

\subsection{A General Result}

Following e.g.~\cite{sin:96}, one can generalise some of the above results for arbitrary penalty functions.
Assume that we consider Information Criterion (IC) of the form
\begin{align}
\label{eq:gIC}
\mathrm{IC}_n(\mathcal{M}_k)  =  -\ell_{\hat{\theta}_{n}^{k}}(y_{0:n-1})
+ pen_n(k),  
\end{align}
for a penalty function $pen_n(k)\in\mathbb{R}$, (strictly) increasing in $k\ge 1$.

\begin{prop}
\label{prop:Gstrong}
\begin{itemize}
\item[(i)]
 If the following hold, for $k'>k\ge 1$,
\begin{align*}
\lim_{n\rightarrow\infty}\tfrac{pen_{n}(k') - pen_{n}(k)}{n}  =0,\quad  
\lim_{n\rightarrow\infty}\tfrac{pen_{n}(k')-  pen_{n}(k)}{\log\log n} & =+\infty,
\end{align*}
then, under Assumptions \ref{ass:1}-\ref{ass:non}, \ref{ass:nested}, the information criterion $\mathrm{IC}_n$ in (\ref{eq:gIC}) is strongly consistent.
\item[(ii)]
If the following hold, for $k'>k\ge 1$,
\begin{align*}
\lim_{n\rightarrow\infty}\tfrac{pen_{n}(k') - pen_{n}(k)}{n}  =0,\quad  
\lim_{n\rightarrow\infty} pen_{n}(k')-  pen_{n}(k)  =+\infty,
\end{align*}
then, under Assumptions \ref{ass:1}-\ref{ass:non}, \ref{ass:nested}, the Information Criterion $\mathrm{IC}_n$ in (\ref{eq:gIC}) is weakly consistent.
\end{itemize}
\end{prop}
\begin{proof}
(i)\,\,It is an immediate generalisation of  the proof of Proposition 
\ref{prop:strong}.\\[0.1cm]
(ii)\,\,
%
%
First, let us consider the case when $\mathcal{M}_{k}\subset\mathcal{M}_{k^{*}}$. 
Then, for any $\epsilon>0$ we have 
\begin{align*}
\mathbb{P}\,\big[\,&\mathrm{IC}_{n}(\mathcal{M}_{k})-\mathrm{IC}_{n}(\mathcal{M}_{k^{*}})>\epsilon\,\big] &
\\ &=\mathbb{P}\,\big[\,\ell_{\hat{\theta}_{n}^{k^{*}}}(y_{0:n-1})-\ell_{\hat{\theta}_{n}^{k}}(y_{0:n-1})
+(pen_{n}(k)-pen_{n}(k^{*}))>\epsilon\,\big]
\\
& \qquad \rightarrow 1. 
\end{align*}
The limit follows from Proposition \ref{prop:diff}(i), 
as the random variable on the left side of the inequality above 
diverges to $+\infty$ w.p.~1, and a.s.~convergence implies convergence in probability. 
This result implies directly  $\lim_{n\rightarrow \infty}\mathbb{P}\,[\,\hat{k}_{n}\ge k^{*}\,] =1$, 
where $\hat{k}_{n}$ denotes the model index minimising $\mathrm{IC}_n(\mathcal{M}_k)$,
$1\le k \le p$.

Next, we consider the case where $\mathcal{M}_{k^{*}}\subset\mathcal{M}_{k}$.
We have that 
\begin{align}
\mathbb{P}\,\big[\,&\mathrm{IC}_{n}(\mathcal{M}_{k}) \le  \mathrm{IC}_{n}(\mathcal{M}_{k^{*}}) \,\big] \nonumber  &
\\ &=\mathbb{P}\,\big[\,\ell_{\hat{\theta}_{n}^{k}}(y_{0:n-1})-\ell_{\hat{\theta}_{n}^{k^*}}(y_{0:n-1})
\ge (pen_{n}(k)-pen_{n}(k^{*}))\,\big].
\label{eq:proo}
\end{align}
Recall the expression for $\ell_{\hat{\theta}_{n}^{k}}(y_{0:n-1})-\ell_{\hat{\theta}_{n}^{k^*}}(y_{0:n-1})$ implied by Proposition \ref{prop:all}(i).
%
%
 The martingale CLT in Proposition \ref{prop:all}(iii) gives that 
\begin{align*}
\tfrac{\nabla
\ell_{\theta_*^k}(y_{0:n-1})}{\sqrt{n}} \Rightarrow 
 N_{d_k}\big( 0, \mathcal{J}_{\theta_*^k}    \big),
\end{align*}
with $\mathcal{J}_{\theta_*^k}$ defined in the obvious manner. 
Let $Z\sim N_{d_k}\big( 0,\mathcal{J}_{\theta_*^k}    \big)$;
the continuous mapping theorem gives that 
\begin{align*}
 \ell_{\hat{\theta}_{n}^{k}}(y_{0:n-1})&-\ell_{\hat{\theta}_{n}^{k^*}}(y_{0:n-1}) \Rightarrow \,\, 
\\  & = 
\tfrac{1}{2}
Z^{\top} \mathcal{J}_{\theta_*^k}^{-1} Z
 - 
\tfrac{1}{2}Z^{\top}_{1:d_{k^*}}\mathcal{J}_{\theta_*}^{-1} Z_{1:d_{k^*}}
=: Z_0
%
\end{align*}
Continuing from (\ref{eq:proo}), 
since $|Z_0|<\infty$, w.p.~1,
for any $\epsilon>0$ we can have some  $n_0$ so that 
for all  $n_1\ge n_0$, $\mathbb{P}\,[\,Z_0\ge pen_{n_1}(k)-pen_{n_1}(k^{*})\,]<\epsilon$. 
Thus, for all $n$ large enough, we have  
that 
\begin{align*}
\mathbb{P}\,\big[\,&\ell_{\hat{\theta}_{n}^{k}}(y_{0:n-1})-\ell_{\hat{\theta}_{n}^{k^*}}(y_{0:n-1})
\ge (pen_{n}(k)-pen_{n}(k^{*}))\,\big] \\
&\le \mathbb{P}\,\big[\,\ell_{\hat{\theta}_{n}^{k}}(y_{0:n-1})-\ell_{\hat{\theta}_{n}^{k^*}}(y_{0:n-1})
\ge (pen_{n_1}(k)-pen_{n_1}(k^{*}))\,\big]  \\
& \qquad \rightarrow \mathbb{P}\,[\,Z_0\ge pen_{n_1}(k)-pen_{n_1}(k^{*})\,]<\epsilon.
\end{align*}
Thus, 
we conclude that 
$\mathbb{P}\,[\,\mathrm{IC}_{n}(\mathcal{M}_{k}) \le  \mathrm{IC}_{n}(\mathcal{M}_{k^{*}})]\rightarrow 0$, so that we have obtained  $\lim_{n\rightarrow \infty}\mathbb{P}\,[\,\hat{k}_{n}\le k^{*}\,] =1$. 

Overall, we have shown that 
$\lim_{n\rightarrow \infty}\mathbb{P}\,[\,\hat{k}_{n}= k^{*}\,] =1$.
%
\end{proof}

\noindent The above results highlight that
the penalty function should 
grow to infinity with the sample size (at certain rate)
to deliver strongly or weakly
consistent~IC.  

\section{Particle Approximation of AIC and BIC}
\label{sec:algorithm}
BIC and AIC can be used for model selection for HMMs but
are typically impossible to calculate analytically
due to intractability of the likelihood function for general
HMMs. 
Thus,
an approximation technique is required. 
We adopt the computational approach of \cite{poyi:11} 
which -- for completeness -- we briefly review in this section.
It involves a particle filtering algorithm coupled with a recursive construction 
for an integral approximation. 

The description  follows closely \cite{poyi:11}. 
The marginal Fisher identity
 gives, %
\begin{align*}
\nabla \ell_{\theta}(y_{0:n}) 
 =\int_{\mathsf{X}}\nabla \log p_{\theta}(x_{n},y_{0:n}) p_{\theta}(x_{n}|y_{0:n})\mu(dx_{n}).
\end{align*}
%
At step $n$, let $(x_{n}^{(i)},w_{n}^{(i)})_{i=1}^{N}$ be a particle approximation of 
$p_{\theta}(dx_{n}| y_{0:n})$, with standardised weights, i.e.~$\sum_i w_{n}^{(i)}=1$, obtained via some particle filtering algorithm,
so that, 
\begin{align}
\nabla \ell_{\theta}(y_{0:n}) 
 \approx \sum_{i=1}^{N}w_n^{(i)}\nabla \log p_{\theta}(x_{n}^{(i)},y_{0:n}). 
\label{eq:rec1}
\end{align}
%
%
We explore the unknown quantity $\nabla \log p_{\theta}(x_{n},y_{0:n})$.
We have,
\begin{align}
p_{\theta}(x_{n},&y_{0:n}) \nonumber \\ 
&=p_{\theta}(y_{0:n-1}) g_{\theta}(y_{n}|x_{n})\int_{\mathsf{X}} q_{\theta}(x_{n}|x_{n-1})p_{\theta}(x_{n-1}| y_{0:n-1})\mu(dx_{n-1}). \label{eq:rec2}
\end{align}
This implies that, 
\begin{align}
&\nabla p_{\theta}(x_{n},y_{0:n})  =
  p_{\theta}(y_{0:n-1})g_{\theta}(y_{n}|x_n)\int_{\mathsf{X}} q_{\theta}(x_{n}|x_{n-1})p_{\theta}(x_{n-1}| y_{0:n-1})\times \nonumber  \\
 & \!\!\!\!\!\Big\{ \nabla \log g_{\theta}(y_{n}|x_{n})+\nabla \log q_{\theta}(x_{n}|x_{n-1})+\nabla \log p_{\theta}(x_{n-1},y_{0:n-1})\Big\} \mu(dx_{n-1}). \label{eq:rec3}
\end{align}
At step $n-1$, let $(x_{n-1}^{(j)},w_{n-1}^{(j)})_{j=1}^{N}$ be a particle approximation of 
the filtering distribution $p_{\theta}(dx_{n-1}| y_{0:n-1})$ and  $(\alpha_{n-1}^{(j)})_{j=1}^{N}$ a sequence of approximations to $(\nabla\log p_{\theta}(x_{n-1}^{(j)},y_{0:n-1}))_{j=1}^{N}$.
Equations (\ref{eq:rec2}), (\ref{eq:rec3}) suggest the following recursive approximation 
of $\nabla\log  p_{\theta}(x_{n}^{(i)},y_{0:n})$, for $1\le i \le N$,
\begin{align}
\alpha_{n}^{(i)} & =\frac{\sum_{j=1}^{N}w_{n-1}^{(j)}q_{\theta}(x_{n}^{(i)}|x_{n-1}^{(j)})}{\sum_{j=1}^{N}w_{n-1}^{(j)}q_{\theta}(x_{n}^{(i)}|x_{n-1}^{(j)})} \times \nonumber  \\[0.2cm] & \qquad \qquad \big\{ \nabla\log g_{\theta}(y_{n}|x_n^{(i)})+\nabla\log q_{\theta}(x_{n}^{(i)}|x_{n-1}^{(j)})+\alpha_{n-1}^{(j)}\big\} , \label{eq:rec4}
\end{align}
%
%
Thus, from (\ref{eq:rec1}), one obtains an estimate of the score function at step $n$, as
\begin{align}
\label{eq:uscore}
\nabla \ell_{\theta}(y_{0:n})  \approx \sum_{i=1}^{N}w_{n}^{(i)}\alpha_{n}^{(i)}.
\end{align}
The calculation in (\ref{eq:rec4}), and the adjoining particle filtering algorithm, can be applied 
recursively
to provide the approximation of the score function in (\ref{eq:uscore}) for $n=0,1,\ldots$.
%
Note that the computational cost of this algorithm is $\mathcal{O}(N^{2})$, 
but is robust for increasing $n$ as it is based on the approximation of the 
filtering distributions rather than the smoothing ones, see results and comments on this point in \cite{poyi:11}. 

Moreover, \cite{poyi:11} use the score function estimation methodology
to propose an \emph{online} gradient ascent algorithm 
for obtaining an MLE-type parameter estimate. 
In more detail, the method is based on the recursion
%
\begin{align}
\theta_{n+1} & =\theta_{n}+\gamma_{n+1}\nabla \log p_{\theta}(y_{n}| y_{0:n-1})\big\vert_{\theta=\theta_{n}},
\end{align}
where $\{\gamma_{k}\}_{k\geq1}$ is a positive decreasing sequence
with 
\begin{align*}
\sum_{k=1}^{\infty}\gamma_{k}=\infty,\quad \sum_{k=1}^{\infty}\gamma_{k}^{2}<\infty.
\end{align*}
To deduce an online algorithm -- following ideas in \cite{legl:97} -- intermediate 
quantities involved in the recursions in (\ref{eq:rec1})-(\ref{eq:rec4}) are calculated at 
different, consecutive parameter values. 
See \cite{poyi:11} for more details, and \cite{legl:97,tadi:18} for analytical studies on the 
convergence properties of the algorithm. In particular, under strict conditions
the algorithm is shown to converge to the maximiser $\theta_*$
of the limiting function of $\theta\mapsto \ell_n(\theta)/n$, as $n\rightarrow\infty$.

\begin{rem}
\label{rem:proxy}
In our setting, we want to use the numerical studies to illustrate the theoretical results obtained for AIC and BIC, so we will use the outcome of the online recursion as  proxy for the MLE. Then,  
the AIC and BIC will be approximated by running a particle filter for the chosen MLE value to obtain an approximation of the log-likelihood of the data at this parameter value.
\end{rem}

%

\section{Empirical Study}
\label{sec:numerics}

We consider the following stochastic
volatility model  (labeled as $\mathcal{SV}$)
\begin{equation*}
\mathcal{SV}:\,\,\, \left\{ \begin{array}{l} X_{t} =\phi X_{t-1}+W_{t}, \\
\,Y_{t}  =\exp(X_{t}/2)V_{t}, \quad t\ge 1, \end{array} \right.
\end{equation*}
and the one with jumps (labeled as $\mathcal{SVJ}$),
\begin{equation*}
\mathcal{SVJ}:\,\,\, \left\{ \begin{array}{l}X_{t}  =\phi X_{t-1}+W_{t}\\
\,Y_{t}  =\exp(X_{t}/2)V_{t}+q_{t}J_{t},\quad t\ge 1, \end{array} \right.
\end{equation*}
where $W_{t}\sim N(0,\sigma_{X}^{2})$, $V_{t}\sim N(0,1)$, $J_{t}\sim N(0,\sigma_{J}^{2})$
and $q_{t}\sim Bernoulli(p)$, all variables assumed independent over the time index $t\ge 1$. 
In both cases $X_0=0$.
The extended model $\mathcal{SVJ}$ can be used to capture instantaneous big jumps 
in the relative changes of the values of the underlying asset; 
the choice of models has been motivated by their use in the literature, see e.g.~\cite{pitt:14}.
Figure
\ref{fig:data} shows two sets of  $n=10^4$  simulated observations, one from $\mathcal{SV}$ and one from 
$\mathcal{SVJ}$, under the 
corresponding true parameter values 
\begin{align*}
(\phi,\,\sigma_{X})=(0.9,\,\sqrt{0.3});\quad (\phi,\,\sigma_{X},\,\sigma_{J},\,p)=(0.9,\,\sqrt{0.3},\,\sqrt{0.6},\,0.6).
\end{align*}
These simulated data will be used in the experiments 
that follow. Scenario~1 (resp.~Scenario 2) corresponds to the case when the true model is $\mathcal{SVJ}$ (resp.~$\mathcal{SV}$). We will compare the two  models, using AIC and BIC, 
in both Scenarios.
The estimated parameter values for $\mathcal{SV}$ and $\mathcal{SVJ}$ -- and the estimates for AIC and BIC using a particle filter --
are obtained via the method of \cite{poyi:11},
reviewed in Section \ref{sec:algorithm}. 
Note that, as we have established in this work, BIC is expected to be consistent for both Scenarios, whereas 
AIC  only for the first Scenario.

\begin{figure}[!h]
\centering
\includegraphics[scale=0.30,angle=270]{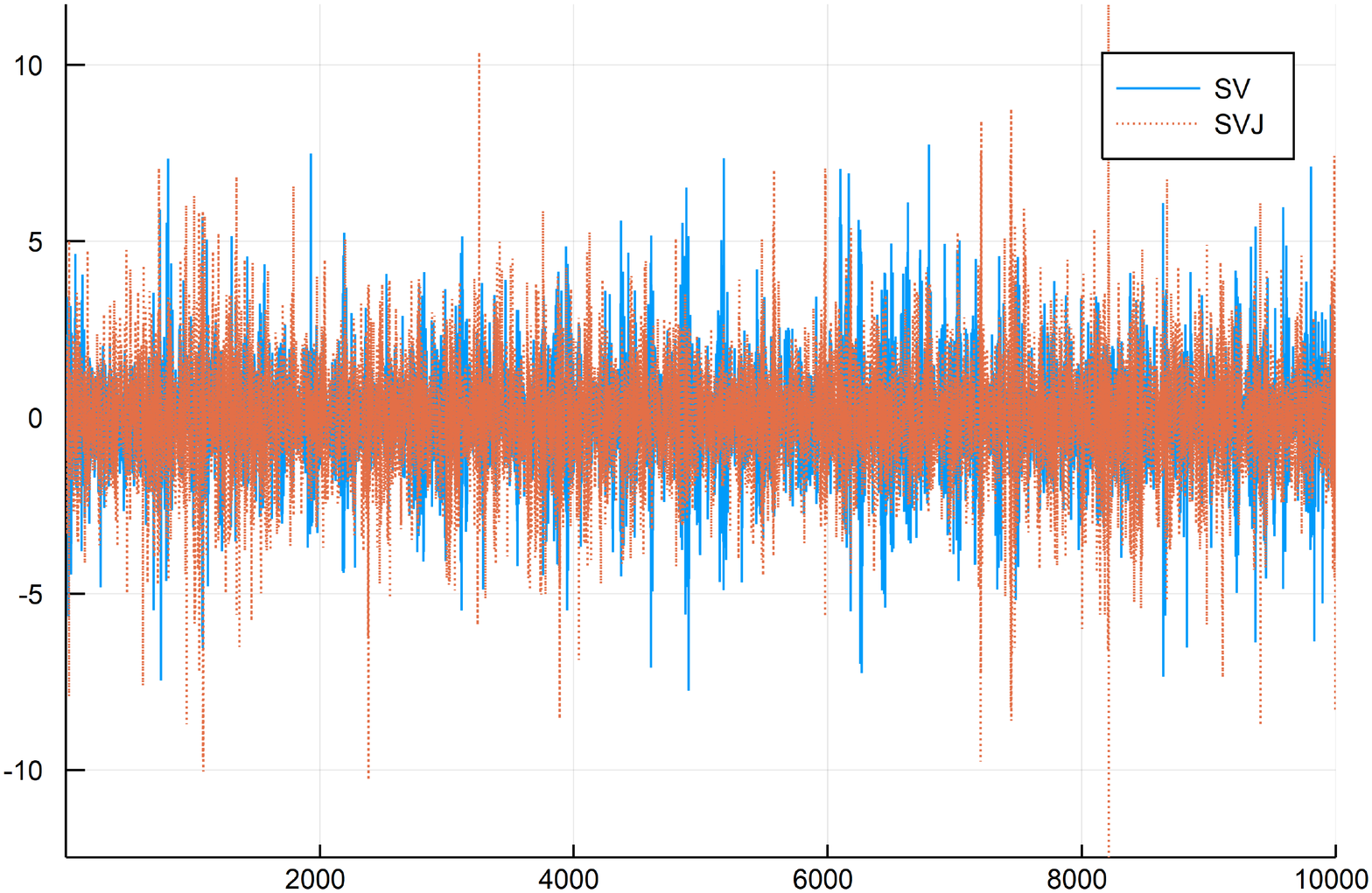}
\vspace{-0.2cm}
\caption{The $n=10^4$ simulated observations from models $\mathcal{SV}$, $\mathcal{SVJ}$, with parameter values $(\phi,\,\sigma_{X})=(0.9,\,\sqrt{0.3})$, $(\phi,\,\sigma_{X},\,\sigma_{J},\,p)=(0.9,\,\sqrt{0.3},\,\sqrt{0.6},\,0.6)$, respectively, to be used in the experiments.}
\label{fig:data}
\end{figure}

We set $\gamma_{n} = n^{-2/3}$ for all numerical experiments in the sequel. Figure \ref{fig:trace} shows estimated parameter values for $\mathcal{SV}$, 
$\mathcal{SVJ}$, 
for both simulated datasets, 
sequentially in the data size, using the online version of the  method
of Section \ref{sec:algorithm}, with $N=200$ particles. 
(We tried also larger number of particles, 
with similar results.)
To further investigate the stability of the algorithm in Section \ref{sec:algorithm}
we summarize in Figures \ref{fig:box1},  \ref{fig:box2} estimates of AIC and BIC for the two models from $R=200$ replications of the same algorithm, for different data sizes.
Figure~\ref{fig:box1} corresponds to Scenario 1
and Figure \ref{fig:box2} to Scenario~2. 
The results obtained seem to indicate that the numerical algorithm used for approximating AIC and BIC is fairly robust in all cases. Also, it appears that in the challenging Scenario~     2, 
even with $n=10^4$ observations, the boxplots do not seem to provide any decisive evidence in favor of 
true model $\mathcal{SV}$.

\begin{figure}[!h]
\centering
\includegraphics[scale=0.38, angle=270]{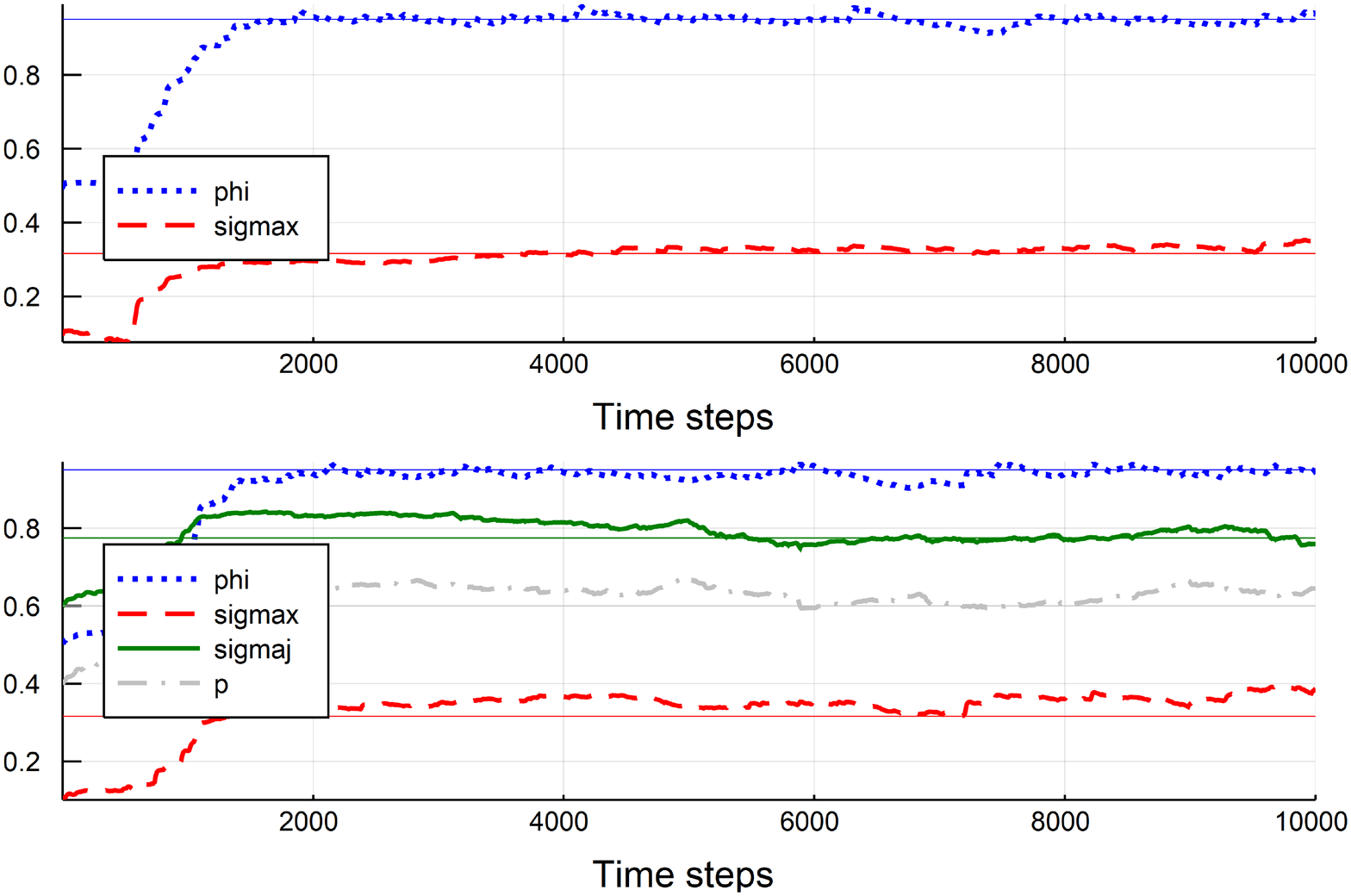}
\vspace{-0.2cm}
\caption{Estimated parameters for the $\mathcal{SV}$ (top panel) and $\mathcal{SVJ}$ (bottom panel) models as obtained -- sequentially in time -- via 
the data simulated from the $\mathcal{SV}$ (top panel) and $\mathcal{SVJ}$ (bottom panel) models respectively and the algorithm reviewed in Section \ref{sec:algorithm} with $N=200$ particles. The horizontal lines indicate the true parameter values in each case.}
\label{fig:trace}
\end{figure}

\begin{figure}[!h]
\centering
\includegraphics[scale=0.28, angle=270]{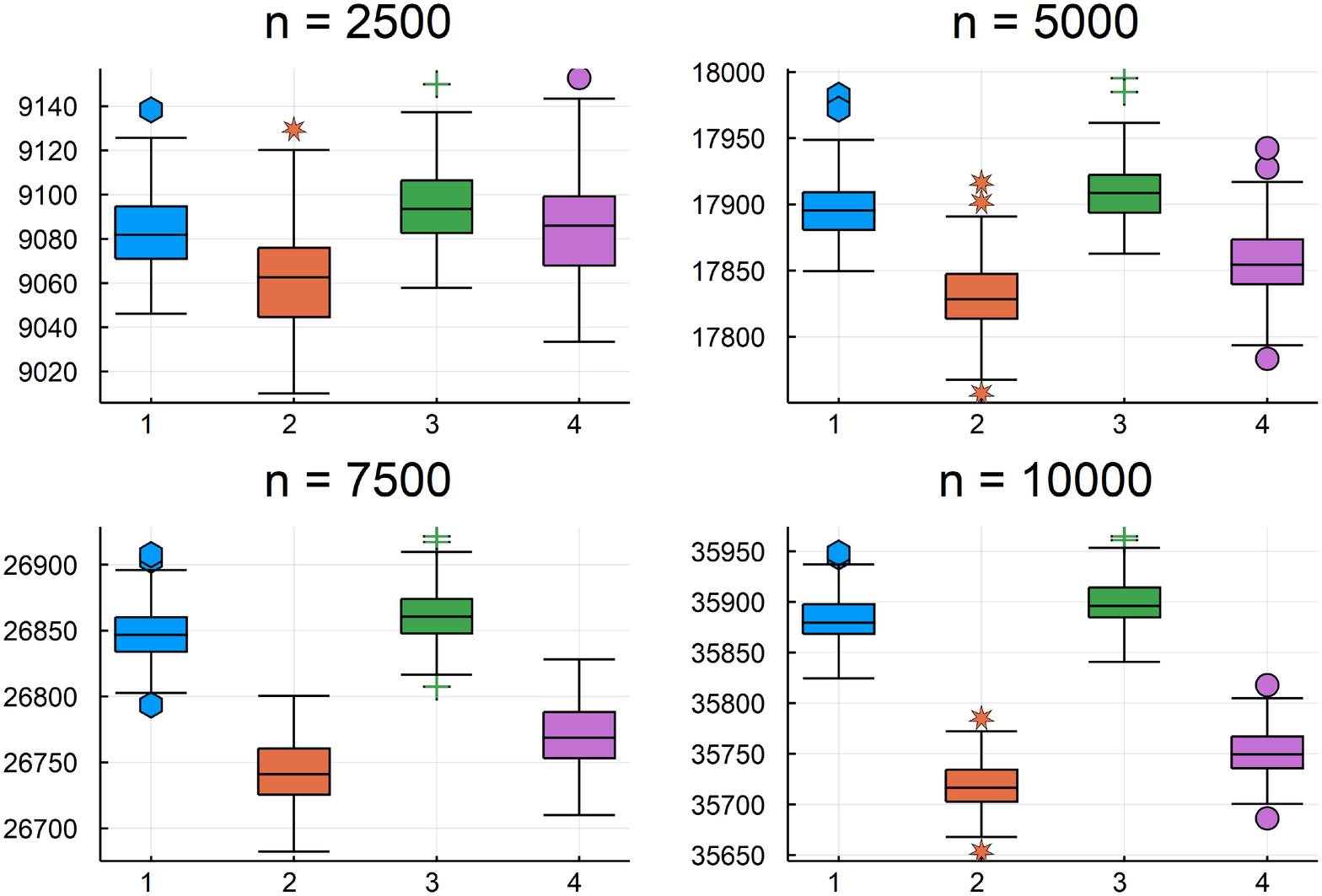}
\caption{Boxplots for Scenario 1 ($\mathcal{SVJ}$ model is true) from $R=200$ estimates ($R$ denotes the replications of the numerical algorithm) of an Information Criterion (IC) and various observation sizes. Blue: $\mathrm{AIC}(\mathcal{SV})$, Orange: $\mathrm{AIC}(\mathcal{SVJ})$, Green:
$\mathrm{BIC}(\mathcal{SV})$, Purple: $\mathrm{BIC}(\mathcal{SVJ})$.}
\label{fig:box1}
\end{figure}

\begin{figure}[!h]
\centering
\includegraphics[scale=0.28, angle=270]{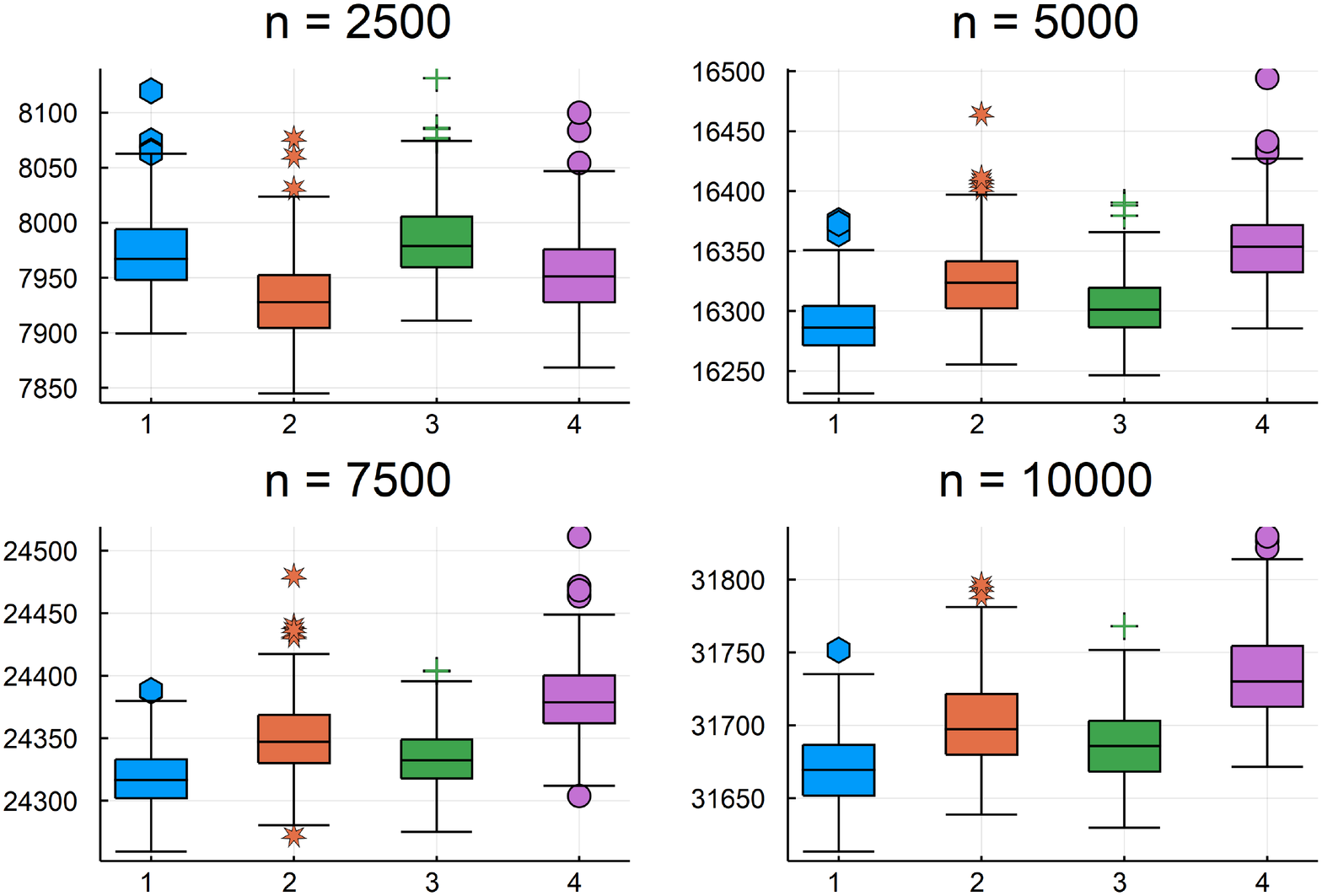}
\caption{Boxplots for Scenario 2 ($\mathcal{SV}$ model is true) from $R=200$ estimates 
 of  an IC and various observation sizes. Blue: $\mathrm{AIC}(\mathcal{SV})$, Orange: $\mathrm{AIC}(\mathcal{SVJ})$, Green:
$\mathrm{BIC}(\mathcal{SV})$, Purple: $\mathrm{BIC}(\mathcal{SVJ})$.}
\label{fig:box2}
\end{figure}

%
%
%
Table \ref{tb:study} shows results from the same $R=200$ replications for the estimation of AIC and BIC for each of the two models and two simulated datasets. 
In agreement with our theory, BIC appears more robust (than AIC) at choosing the correct model for the dataset simulated from  $\mathcal{SV}$.
Figure \ref{fig:best} plots differences in AIC and BIC in Scenario 2, 
sequentially in the data size -- more accurately, a proxy of the differences, see Remark 
\ref{rem:proxy}.
To be precise, the blue line denotes the `path' of $\mathrm{AIC}(\mathcal{SV})-\mathrm{AIC}(\mathcal{SVJ})$,
and the red line denotes the one of $\mathrm{BIC}(\mathcal{SV})-\mathrm{BIC}(\mathcal{SVJ})$.
Since model $\mathcal{SV}$ is true in this case, the
difference should be lower than zero for large enough $n$ if the used IC
were consistent. As one can see, the difference in BIC is
always negative after a large enough sample size $n$. In contrast,
and in agreement with our theory, the difference in AIC never has such property.
For instance, sometime after $n=10^4$, the difference increased and
exceeded the zero line. This is a clear empirical manifestation
of Proposition \ref{prop:clear}; so, whereas in the previous plots the deficiency of 
AIC was difficult to showcase when looking at \emph{fixed} data sizes, 
such deficiency became clear when we look at the evolution of AIC as a function
of data size.

\begin{table}[!h]
\centering%
\begin{tabular}{|l|r|r|r|r|}
\hline 
$n$ & $2,500$ & $5,000$ & $7,500$ & $10,000$\tabularnewline
\hline 
\hline 
$\mathrm{AIC}(\mathcal{SV})$ & $32/200$ & $4/200$ & $0/200$ & $0/200$\tabularnewline
\hline 
$\mathrm{AIC}(\mathcal{SVJ})$ & $168/200$ & $196/200$ & $200/200$ & $200/200$\tabularnewline
\hline 
$\mathrm{BIC}(\mathcal{SV})$ & $51/200$ & $5/200$ & $0/200$ & $0/200$\tabularnewline
\hline 
$\mathrm{BIC}(\mathcal{SVJ})$ & $149/200$ & $195/200$ & $200/200$ & $200/200$\tabularnewline
\hline 
\end{tabular}%
\vspace{0.2cm}

\begin{tabular}{|l|r|r|r|r|}
\hline 
$n$ & $2,500$ & $5,000$ & $7,500$ & $10,000$\tabularnewline
\hline 
\hline 
$\mathrm{AIC}(\mathcal{SV})$ & $154/200$ & $161/200$ & $153/200$ & $158/200$\tabularnewline
\hline 
$\mathrm{AIC}(\mathcal{SVJ})$ & $46/200$ & $39/200$ & $47/200$ & $42/200$\tabularnewline
\hline 
$\mathrm{BIC}(\mathcal{SV})$ & $179/200$ & $173/200$ & $184/200$ & $192/200$\tabularnewline
\hline 
$\mathrm{BIC}(\mathcal{SVJ})$ & $21/200$ & $27/200$ & $16/200$ & $8/200$\tabularnewline
\hline 
\end{tabular}
\caption{Results after $R=200$ replications of the approximation algorithm 
with $N=200$ particles. The top (resp.~bottom) table shows results for the data simulated from 
$\mathcal{SVJ}$ (resp.~$\mathcal{SV}$).
The 1st (resp.~2nd) row in each table shows the fraction of the replications  where the estimated AIC is smaller for the $\mathcal{SV}$ model (resp.~$\mathcal{SVJ}$ model) for different   data sizes; rows 3 and 4 show similar results for  BIC.  }
\label{tb:study}
\end{table}

%

\begin{figure}[!h]
\centering
\includegraphics[scale=0.4, angle=270]{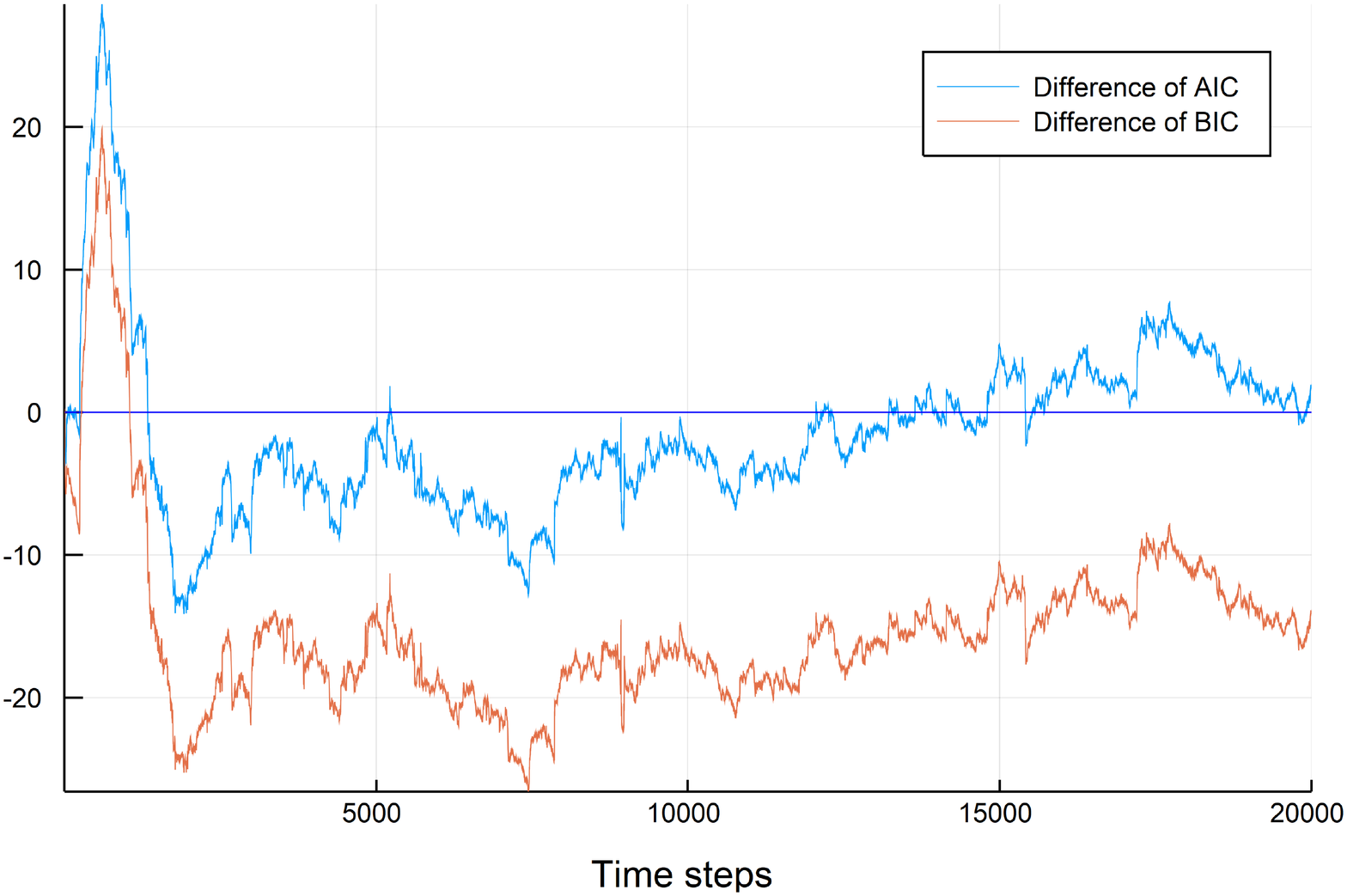}
\caption{The path of differences in AIC and BIC in Scenario 2 ($\mathcal{SV}$ is the true model). That is, the blue line show the approximated value of  $\mathrm{AIC}(\mathcal{SV})-\mathrm{AIC}(\mathcal{SVJ})$ as a function of data size,
and the red line the corresponding function for $\mathrm{BIC}(\mathcal{SV})-\mathrm{BIC}(\mathcal{SVJ})$.}
\label{fig:best}
\end{figure}

\section{Conclusions and Remarks}
\label{sec:final}

We have investigated the asymptotic behaviour of BIC, the evidence and AIC for nested HMMs,
and have derived new results concerning their consistency properties. Our work shows that BIC -- and the evidence -- are strongly consistent for a general class of HMMs. In contrast, for a similarly posed Model Selection problem,  AIC is not even weakly consistent. Our study focuses on asymptotics for increasing data size, 
so we do not investigate finite sample-size results for
BIC, evidence and AIC, such as optimality properties.
It is well-known that AIC is minimax-rate optimal  but BIC is not in many cases,
see e.g.~\cite{barr:99}. We conjecture this might also be the case for general~HMMs.

The technique of constructing stationary,
ergodic processes 
by introducing a backward infinite extension of the
observations -- see (\ref{eq:extend}) in Section   \ref{sec:asy} --
has been used in many other studies,  even beyond HMMs.
E.g.~\citet{douc2020posterior} use this approach to study
posterior consistency for a class of partially observed Markov models; \citet{lehericy2018nonasymptotic} use it
to investigate non-asymptotic behaviour of
MLE for (finite state space) HMMs; 
\cite{sylv:13} apply the technique within an online EM setting for HMMs;
\cite{diel:20} consider more general classes of latent variable models.

We note that asymptotic results about 
the MLE for HMMs  have recently been obtained under weaker conditions. 
See, e.g., \citet{douc2011consistency, douc2016maximizing}
for developments that go beyond compact spaces. 
Here, we have worked with strict assumptions on the state space (see -- indicatively -- Assumption \ref{ass:3}, Section \ref{sec:asy}), so that we obtain an important first set of illustrative results for Model Selection Criteria, avoiding at the same time an overload of technicalities. Future investigations are expected to further weaken the conditions we have used here. 

There are challenges when trying to move beyond the Model-Correctness setting.
As we have described in the first parts of the paper, \citet{douc2012asymptotic}
show that the MLE converges a.s.~even for misspecified models under
mild assumptions. However, a CLT for the MLE in the context of general
state-space misspecified HMMs has yet to been proven. To the best
of our knowledge, only \citet{pouzo2016maximum} obtain such a result
for a finite state space X. Thus, extending our results to non-nested
settings or/and ones where one does not assume correctness of a model,
is a non-trivial undertaking that requires extensive further research. Also, we note that AIC is asymptotically prediction efficient in some
misspecified models whilst BIC is not. The above discussion suggests that
investigating asymptotic behaviour of model selection criteria under
No-Model-Correctness for general HMM models is an important open problem that requires further research.



One can use alternative numerical algorithms instead of the one we have used here, 
and describe in Section \ref{sec:algorithm} -- see e.g.~the approach in \cite{olss:17}. 
Note that the numerical algorithm used in the paper is mostly a tool for illustrating our theoretical findings, which is the main focus of our work. The numerical study  shown in the paper already delivers the points stemming from the theory, so we have refrained from describing/implementing further methods to avoid diverting attention from our main findings.

From a practitioner point of view, our results and numerical study indicate that AIC can wrongly select the more complex model due to ineffective penalty term. Critically, this can be difficult 
to assess using standard experiments. Our study has shown that one needs to investigate the evolution of AIC against data size to clearly highlight its deficiency in the context of a numerical study. 
We stress here that in the numerical experiment we have knowingly used models  for which several of the stated Assumptions will not hold (maybe most notably, the strong mixing Assumption \ref{ass:3}). The aim is to illustrate at least numerically, that while our assumptions are standard in the literature, they serve for simplifying the path to otherwise too technical derivations and provide results that are expected to hold in much more general settings.

\section*{Acknowledgments}

SY was supported by the Alan Turing Institute for Data Sciences under
grant number TU/C/000013. AB was supported by the Leverhulme Trust Prize. We thank the three anonymous referees and the Associate Editor for comments, suggestions that have greatly improved the content of the paper.

\bibliographystyle{apalike}
\bibliography{references}

\begin{thebibliography}{}

\bibitem[Akaike, 1974]{akai:74}
Akaike, H. (1974).
\newblock A new look at the statistical model identification.
\newblock {\em IEEE Transactions on Automatic Control}, 19(6):716--723.

\bibitem[Barron et~al., 1999]{barr:99}
Barron, A., Birg{\'e}, L., and Massart, P. (1999).
\newblock Risk bounds for model selection via penalization.
\newblock {\em Probability theory and related fields}, 113(3):301--413.

\bibitem[Capp{\'e} et~al., 2005]{capp:05}
Capp{\'e}, O., Moulines, E., and Ryd{\'e}n, T. (2005).
\newblock {\em Inference in Hidden {M}arkov Models. {S}pringer Series in
  Statistics}.
\newblock Springer.

\bibitem[Chatterjee et~al., 2018]{chat:18}
Chatterjee, D., Maitra, T., and Bhattacharya, S. (2018).
\newblock A short note on almost sure convergence of {B}ayes factors in the
  general set-up.
\newblock {\em The American Statistician}, pages 1--4.

\bibitem[Claeskens and Hjort, 2008]{clae:08}
Claeskens, G. and Hjort, N.~L. (2008).
\newblock {\em Model {S}election and {M}odel {A}veraging}, volume 330.
\newblock Cambridge University Press, Cambridge.

\bibitem[Crouse et~al., 1998]{crouse1998wavelet}
Crouse, M., Nowak, R., and Baraniuk, R. (1998).
\newblock Wavelet-based statistical signal processing using hidden {M}arkov
  models.
\newblock {\em IEEE Transactions on Signal Processing}, 46(4):886--902.

\bibitem[Csisz{\'a}r and Shields, 2000]{csiszar2000consistency}
Csisz{\'a}r, I. and Shields, P. (2000).
\newblock The consistency of the {BIC} {M}arkov order estimator.
\newblock {\em The Annals of Statistics}, 28(6):1601--1619.

\bibitem[Del~Moral, 2004]{del:04}
Del~Moral, P. (2004).
\newblock {\em Feynman--{K}ac Formulae: {G}enealogical and Interacting Particle
  Systems with Applications}.
\newblock Springer.

\bibitem[Del~Moral, 2013]{del:13}
Del~Moral, P. (2013).
\newblock {\em Mean Field Simulation for {M}onte {C}arlo Integration}.
\newblock Chapman and Hall/CRC.

\bibitem[Diel et~al., 2020]{diel:20}
Diel, R., Le~Corff, S., and Lerasle, M. (2020).
\newblock Learning the distribution of latent variables in paired comparison
  models with round-robin scheduling.
\newblock {\em arXiv preprint arXiv:1707.0136}.

\bibitem[Ding et~al., 2017]{ding2017bridging}
Ding, J., Tarokh, V., and Yang, Y. (2017).
\newblock Bridging {AIC} and {BIC}: a new criterion for autoregression.
\newblock {\em IEEE Transactions on Information Theory}, 64(6):4024--4043.

\bibitem[Ding et~al., 2018]{ding2018model}
Ding, J., Tarokh, V., and Yang, Y. (2018).
\newblock Model selection techniques: An overview.
\newblock {\em IEEE Signal Processing Magazine}, 35(6):16--34.

\bibitem[Douc and Moulines, 2012]{douc2012asymptotic}
Douc, R. and Moulines, E. (2012).
\newblock Asymptotic properties of the maximum likelihood estimation in
  misspecified hidden {M}arkov models.
\newblock {\em The Annals of Statistics}, 40(5):2697--2732.

\bibitem[Douc et~al., 2011]{douc2011consistency}
Douc, R., Moulines, E., Olsson, J., and Van~Handel, R. (2011).
\newblock Consistency of the maximum likelihood estimator for general hidden
  {M}arkov models.
\newblock {\em The Annals of Statistics}, 39(1):474--513.

\bibitem[Douc et~al., 2004]{douc:04}
Douc, R., Moulines, E., and Ryd{\'e}n, T. (2004).
\newblock Asymptotic properties of the maximum likelihood estimator in
  autoregressive models with {M}arkov regime.
\newblock {\em The Annals of Statistics}, 32(5):2254--2304.

\bibitem[Douc et~al., 2014]{douc:14}
Douc, R., Moulines, E., and Stoffer, D. (2014).
\newblock {\em Nonlinear Time Series: Theory, Methods and Applications with {R}
  Examples}.
\newblock CRC Press.

\bibitem[Douc et~al., 2020]{douc2020posterior}
Douc, R., Olsson, J., and Roueff, F. (2020).
\newblock Posterior consistency for partially observed {M}arkov models.
\newblock {\em Stochastic Processes and their Applications}, 130(2):733--759.

\bibitem[Douc et~al., 2016]{douc2016maximizing}
Douc, R., Roueff, F., and Sim, T. (2016).
\newblock The maximizing set of the asymptotic normalized log-likelihood for
  partially observed {M}arkov chains.
\newblock {\em The Annals of Applied Probability}, 26(4):2357--2383.

\bibitem[Gales and Young, 2008]{gales2008application}
Gales, M. and Young, S. (2008).
\newblock The application of hidden {M}arkov models in speech recognition.
\newblock {\em Foundations and Trends in Signal Processing}, 1(3):195--304.

\bibitem[Gassiat and Boucheron, 2003]{gass:03}
Gassiat, E. and Boucheron, S. (2003).
\newblock Optimal error exponents in hidden {M}arkov models order estimation.
\newblock {\em IEEE Transactions on Information Theory}, 49(4):964--980.

\bibitem[Green and Richardson, 2002]{green2002hidden}
Green, P. and Richardson, S. (2002).
\newblock Hidden {M}arkov models and disease mapping.
\newblock {\em Journal of the American Statistical Association},
  97(460):1055--1070.

\bibitem[Hall and Heyde, 1980]{hall:80}
Hall, P. and Heyde, C.~C. (1980).
\newblock {\em Martingale Limit Theory and its Application}.
\newblock Academic Press.

\bibitem[Hurvich and Tsai, 1989]{hurvich1989regression}
Hurvich, C. and Tsai, C.-L. (1989).
\newblock Regression and time series model selection in small samples.
\newblock {\em Biometrika}, 76(2):297--307.

\bibitem[Hurvich and Tsai, 1990]{hurvich1990impact}
Hurvich, C. and Tsai, C.-L. (1990).
\newblock The impact of model selection on inference in linear regression.
\newblock {\em The American Statistician}, 44(3):214--217.

\bibitem[Ibragimov and Sharakhmetov, 1999]{ibra:99}
Ibragimov, R. and Sharakhmetov, S. (1999).
\newblock Analogues of {K}hintchine, {M}arcinkiewicz--{Z}ygmund and {R}osenthal
  inequalities for symmetric statistics.
\newblock {\em Scandinavian Journal of Statistics}, 26(4):621--633.

\bibitem[Jeffreys, 1998]{jeffreys1998theory}
Jeffreys, H. (1998).
\newblock {\em The Theory of Probability}.
\newblock OUP Oxford.

\bibitem[Kass and Raftery, 1995]{kass1995bayes}
Kass, R. and Raftery, A. (1995).
\newblock Bayes {F}actors.
\newblock {\em Journal of the American Statistical Association},
  90(430):773--795.

\bibitem[Kass et~al., 1990]{kass:90}
Kass, R.~E., Tierney, L., and Kadane, J.~B. (1990).
\newblock The validity of posterior expansions based on {L}aplace's method.
\newblock In {\em In Bayesian and Likelihood Methods in Statistics and
  Econometrics, edited by S. Geisser, J. S. Hodges, S. J. Press and A.
  Zellner}, pages 473--488. North-Holland Amsterdam.

\bibitem[Konishi and Kitagawa, 1996]{konishi1996generalised}
Konishi, S. and Kitagawa, G. (1996).
\newblock Generalised information criteria in model selection.
\newblock {\em Biometrika}, 83(4):875--890.

\bibitem[Lang, 2012]{lang:2012}
Lang, S. (2012).
\newblock {\em Real and functional analysis}, volume 142.
\newblock Springer Science \& Business Media.

\bibitem[Le~Corff and Fort, 2013]{sylv:13}
Le~Corff, S. and Fort, G. (2013).
\newblock Online expectation maximization based algorithms for inference in
  hidden {M}arkov models.
\newblock {\em Electronic Journal of Statistics}, 7:763--792.

\bibitem[Le~Gland and Mevel, 1997]{legl:97}
Le~Gland, F. and Mevel, L. (1997).
\newblock Asymptotic behaviour of the {MLE} in hidden {M}arkov models.
\newblock In {\em Proceedings of the 4th European Control Conference, Bruxelles
  1997}.

\bibitem[Leh{\'e}ricy, 2018]{lehericy2018nonasymptotic}
Leh{\'e}ricy, L. (2018).
\newblock Nonasymptotic control of the {MLE} for misspecified nonparametric
  hidden {M}arkov models.
\newblock {\em arXiv preprint arXiv:1807.03997}.

\bibitem[Mamon and Elliott, 2007]{mamon2007hidden}
Mamon, R. and Elliott, R. (2007).
\newblock {\em Hidden {M}arkov Models in Finance}, volume 460.
\newblock Springer.

\bibitem[Nishii, 1988]{nish:88}
Nishii, R. (1988).
\newblock Maximum likelihood principle and model selection when the true model
  is unspecified.
\newblock {\em Journal of Multivariate Analysis}, 27(2):392--403.

\bibitem[Olsson and Alenl{\"o}v, 2017]{olss:17}
Olsson, J. and Alenl{\"o}v, J.~W. (2017).
\newblock Particle-based online estimation of tangent filters with application
  to parameter estimation in nonlinear state-space models.
\newblock {\em Annals of the Institute of Statistical Mathematics}, pages
  1--32.

\bibitem[Pitt et~al., 2014]{pitt:14}
Pitt, M., Malik, S., and Doucet, A. (2014).
\newblock Simulated likelihood inference for stochastic volatility models using
  continuous particle filtering.
\newblock {\em Annals of the Institute of Statistical Mathematics},
  66(3):527--552.

\bibitem[Pouzo et~al., 2016]{pouzo2016maximum}
Pouzo, D., Psaradakis, Z., and Sola, M. (2016).
\newblock Maximum likelihood estimation in possibly misspecified dynamic models
  with time inhomogeneous {M}arkov regimes.
\newblock {\em arxiv preprint arXiv:1612.04932}.

\bibitem[Poyiadjis et~al., 2011]{poyi:11}
Poyiadjis, G., Doucet, A., and Singh, S. (2011).
\newblock Particle approximations of the score and observed information matrix
  in state space models with application to parameter estimation.
\newblock {\em Biometrika}, 98(1):65--80.

\bibitem[Schwarz, 1978]{schw:78}
Schwarz, G. (1978).
\newblock Estimating the dimension of a model.
\newblock {\em The Annals of Statistics}, 6(2):461--464.

\bibitem[Shao, 1997]{shao1997asymptotic}
Shao, J. (1997).
\newblock An asymptotic theory for linear model selection.
\newblock {\em Statistica Sinica}, pages 221--242.

\bibitem[Shao et~al., 2017]{shao2017bayesian}
Shao, S., Jacob, P., Ding, J., and Tarokh, V. (2017).
\newblock Bayesian model comparison with the {H}yv\"arinen score: {C}omputation
  and consistency.
\newblock {\em arXiv preprint arXiv:1711.00136}.

\bibitem[Shibata, 1976]{shibata1976selection}
Shibata, R. (1976).
\newblock Selection of the order of an autoregressive model by {A}kaike's
  information criterion.
\newblock {\em Biometrika}, 63(1):117--126.

\bibitem[Shibata, 1980]{shib:80}
Shibata, R. (1980).
\newblock Asymptotically efficient selection of the order of the model for
  estimating parameters of a linear process.
\newblock {\em The annals of statistics}, pages 147--164.

\bibitem[Shibata, 1981]{shib:81}
Shibata, R. (1981).
\newblock An optimal selection of regression variables.
\newblock {\em Biometrika}, 68(1):45--54.

\bibitem[Sin and White, 1996]{sin:96}
Sin, C.-Y. and White, H. (1996).
\newblock Information criteria for selecting possibly misspecified parametric
  models.
\newblock {\em Journal of Econometrics}, 71(1-2):207--225.

\bibitem[Stone, 1977]{stone1977asymptotic}
Stone, M. (1977).
\newblock An asymptotic equivalence of choice of model by cross-validation and
  {A}kaike's criterion.
\newblock {\em Journal of the Royal Statistical Society: Series B
  (Methodological)}, 39(1):44--47.

\bibitem[Stout, 1970]{stou:70}
Stout, W. (1970).
\newblock The {H}artman-{W}intner law of the iterated logarithm for
  martingales.
\newblock {\em The Annals of Mathematical Statistics}, 41(6):2158--2160.

\bibitem[Tadic and Doucet, 2018]{tadi:18}
Tadic, V.~Z. and Doucet, A. (2018).
\newblock Asymptotic properties of recursive maximum likelihood estimation in
  non-linear state-space models.
\newblock {\em arXiv preprint arXiv:1806.09571}.

\bibitem[Takeuchi, 1976]{take:76}
Takeuchi, K. (1976).
\newblock Distribution of information statistics and validity criteria of
  models.
\newblock {\em Mathematical Science}, 153:12--18.

\bibitem[Varin and Vidoni, 2005]{varin2005note}
Varin, C. and Vidoni, P. (2005).
\newblock A note on composite likelihood inference and model selection.
\newblock {\em Biometrika}, 92(3):519--528.

\bibitem[Yang, 2005]{yang2005can}
Yang, Y. (2005).
\newblock Can the strengths of {AIC} and {BIC} be shared? {A} conflict between
  model indentification and regression estimation.
\newblock {\em Biometrika}, 92(4):937--950.

\bibitem[Yoon, 2009]{yoon2009hidden}
Yoon, B.-J. (2009).
\newblock Hidden {M}arkov models and their applications in biological sequence
  analysis.
\newblock {\em Current Genomics}, 10(6):402--415.

\end{thebibliography}

\end{document}